\documentclass[11pt,reqno,letterpaper]{amsart}

\usepackage{amssymb, amsmath, amsthm, amsfonts}
\usepackage{bbm, hyperref, float, comment}

\usepackage[text={160mm,229mm},centering]{geometry}

\usepackage{tikz, tikz-cd}
\usetikzlibrary{matrix,arrows}
\tikzset{commutative diagrams/diagrams={baseline=-2.5pt}} 

\newcommand{\chapter}

\newcommand{\quotes}[1]{\textquoteleft{#1}'}
\newcommand{\set}[1]{\{{#1}\}}
\newcommand{\complex}[1]{\ensuremath{\left\{{#1}\right\}}}
\newcommand{\cone}[1]{\ensuremath{\operatorname{Cone}\left({#1}\right)}}
\newcommand{\setconds}[2]{\ensuremath{\left\{\begin{array}{c|c} #1 & #2 \end{array}}\right\}}

\newcommand\xyhook{\ar@{^{(}->}}
\newcommand\nextcell{\pgfmatrixnextcell}
\newcommand\Z{\mathbb Z}
\newcommand\C{\mathbb C}

\newcommand\iso{\simeq}

\newcommand\id{\operatorname{id}}
\newcommand\acts{\curvearrowright}
\newcommand\into{\hookrightarrow}
\newcommand\To{\longrightarrow}
\newcommand\so{\ \ \Longrightarrow\ }

\newcommand\Hom{\operatorname{Hom}}
\newcommand\Ext{\operatorname{Ext}}
\newcommand\RHom{\operatorname{\mathbb{R}Hom}}
\newcommand\RDerived{\mathbb{R}}
\renewcommand\Im{\operatorname{Im}}

\renewcommand\P{\mathbb P}
\newcommand\PP{\mathbb P}
\newcommand{\Gr}{\mathbbm{Gr}}
\newcommand\pt{\operatorname{pt}}

\theoremstyle{plain}
\newtheorem{thm}{Theorem}
\newtheorem{prop}[thm]{Proposition}
\newtheorem{lem}[thm]{Lemma}
\newtheorem{cor}[thm]{Corollary}
\newtheorem{alg}[thm]{Algorithm}
\newtheorem{keythm}{Theorem}

\theoremstyle{definition}
\newtheorem{defn}[thm]{Definition}
\newtheorem{eg}[thm]{Example}
\newtheorem*{keydefn}{Definition}

\theoremstyle{remark}
\newtheorem{rem}[thm]{Remark}
\newtheorem*{acks}{Acknowledgements}
\newtheorem*{notn}{Notation}

\numberwithin{thm}{section}

\newcommand{\emphasis}[1]{{\em #1}}
\newcommand{\case}[1]{{\em\bf Case #1:}}

\newcommand{\cE}{\mathcal{E}}
\newcommand{\cF}{\mathcal{F}}
\newcommand{\cO}{\mathcal{O}}
\newcommand{\cS}{\mathcal{S}}
\newcommand{\cV}{\mathcal{V}}
\newcommand{\cW}{\mathcal{W}}
\newcommand{\mcF}{\mathcal{F}}
\newcommand{\mcT}{\mathcal{T}}
\newcommand{\mfF}{\mathcal{F}}

\newcommand{\mfP}{\mathcal{P}}
\newcommand{\mfQ}{\mathcal{Q}}
\newcommand{\mfR}{\mathcal{R}}
\newcommand{\mfS}{\mathcal{S}}
\newcommand{\mfT}{\mathcal{T}}
\newcommand{\mfX}{\mathfrak{X}}
\newcommand{\mfY}{\mathcal{Y}}
\newcommand{\mfZ}{\mathcal{Z}}
\newcommand{\bbS}{\mathbb{S}}

\newcommand{\Tot}{\operatorname{Tot}}
\newcommand{\Sym}{\operatorname{Sym}}
\newcommand{\Wedge}{\mbox{\large{$\wedge$}}}
\newcommand{\br}[1]{\langle#1\rangle}
\newcommand{\comp}[3]{\operatorname{Comp}_{#1}^{#2}(#3)}
\newcommand\quot{/\kern-.7ex/}

 \newcounter{enumistar}
\addtocounter{enumistar}{1}

\newcommand{\staircase}[1]{\delta_{#1}}
\newcommand{\cpt}[2]{{#1}^{#2}}

\newcommand{\bijar}[1][]{%
 \ar[#1]
 \ar@<0.7ex>@{}[#1]|-*=0[@]{\sim}} 
\newcommand{\bijarswap}[1][]{%
 \ar[#1]
 \ar@<0.7ex>@{}[#1]|-*=0[@]{\sim}} 

\newcommand{\PlainYoungDiag}[6]{

\def\xStart{#2} \def\yStart{#3}
\def\xSize{#4} \def\ySize{#5}

\draw (\xStart,\yStart) 
  rectangle (\xStart+\xSize,\yStart+\ySize);
\draw[help lines] (\xStart,\yStart) grid (\xStart+\xSize,\yStart+\ySize);

\newcounter{#1row}
\foreach \rowLength in #6 {
  \ifnum\rowLength>0 {
    \foreach \col in {1,...,\rowLength}{
      \draw[fill=gray!50] (\xStart+\col-1,\yStart+\value{#1row}) rectangle (\xStart+\col,\yStart+\value{#1row}+1) ; } }
  \fi
  \stepcounter{#1row} }
}

\newcommand{\YoungDiag}[7]{

\def\xStart{#2} \def\yStart{#3}
\def\xSize{#4} \def\ySize{#5}
\draw   node at (\xStart+\xSize/2,\yStart-0.75) ()  {#7};
\PlainYoungDiag{#1}{#2}{#3}{#4}{#5}{#6}

}

\newcommand{\AnnotatedYoungDiag}[9]{

\def\xStart{#2} \def\yStart{#3}
\def\xSize{#4} \def\ySize{#5}
\def\rhSpace{#9}
\draw   node at (\xStart-\xSize,\yStart+\ySize/2) ()  {#7};
\draw   node at (\xStart+\rhSpace/2*\xSize+3/2*\xSize,\yStart+\ySize/2) ()  {#8};
\PlainYoungDiag{#1}{#2}{#3}{#4}{#5}{#6}

}

\begin{document}

\title[Window shifts and Grassmannian twists]{Window shifts, flop equivalences and Grassmannian twists}%

\author{Will Donovan}
\address{School of Mathematics and Maxwell Institute of Mathematics, University of Edinburgh, Edinburgh, EH9 3JZ, U.K.}
\email{will.donovan@ed.ac.uk}

\author{Ed Segal}
\address{Department of Mathematics, Imperial College London,
London, SW7 2AZ, U.K.}
\email{edward.segal04@imperial.ac.uk}


\begin{abstract} 
We introduce a new class of autoequivalences that act on the derived categories of certain vector bundles over Grassmannians. These autoequivalences arise from Grassmannian flops: they generalize Seidel-Thomas spherical twists, which can be seen as arising from standard flops. We first give a simple algebraic construction, which is well-suited to explicit computations. We then give a geometric construction using spherical functors which we prove is equivalent.
\end{abstract}

\maketitle

\tableofcontents

\section{Introduction}

Derived equivalences corresponding to flops were first explored by Bondal and Orlov \cite{Bondal:1995vra}. They exhibited an equivalence of bounded derived categories of coherent sheaves corresponding to the standard flop of a projective space $\PP^{d-1}$ in a smooth algebraic variety with normal bundle $\mathcal{N}\iso\cO(-1)^{\oplus d}$ \cite[Theorem 3.6]{Bondal:1995vra}. More generally it is conjectured \cite[Conjecture 5.1]{Kawamata:2002vq} that for any flop between smooth projective varieties there exists a derived equivalence. This follows for 3-folds by work of Bridgeland \cite{Bridgeland:2000ux}, but is still an open question in higher dimensions.

Examples of flops, including the standard flop, may be obtained by variation of GIT, and in this case there is a particular approach to constructing derived equivalences. Suppose $X_+$ and $X_-$ are a pair of varieties related by a flop, and that both are possible GIT quotients of a larger space $M$ by the action of a group $G$. Then $X_+$ and $X_-$ are open substacks of the Artin stack $\mfX=[M/G]$, and there are restriction functors from $D^b(\mfX)$ to both $D^b(X_+)$ and $D^b(X_-)$. So one way to construct an equivalence between $D^b(X_+)$ and $D^b(X_-)$ is to find a subcategory inside $D^b(\mfX)$ which is equivalent to both of them. We call such a subcategory a \quotes{window}.

This technique was inspired by the physical analysis carried out by Herbst, Hori and Page in \cite{Herbst:2008wl}, and was introduced into the mathematics literature by the second author in \cite{Segal:2009tua}. Both of these papers were concerned with Landau-Ginzburg models, where the derived category is modified by a superpotential,  however the technique is still interesting when applied to ordinary derived categories.

In this paper we study a particular class of examples, which are local models of \quotes{Grassmannian flops}. For us, $X_+$ is the total space of the vector bundle
$$\Hom(V,S) \To \Gr(r,V)$$
where $S$ is the tautological subspace bundle on the Grassmannian $\Gr(r,V)$ of $r$-dimensional subspaces of a vector space $V$, where $0<r<\dim V$. This can be flopped to a second space $X_-$, which is the total space of a vector bundle over the dual Grassmannian $\Gr(V,r)$. (When $V$ is 2-dimensional, and $r=1$, this is the standard Atiyah flop.) This flop arises from a GIT problem, and we show that it is possible to find a window. In fact we find a whole set of windows, indexed by $\Z$, and hence show:
\begin{keythm}[Theorem \ref{thm.equivalences2}] For $k \in \mathbb{Z}$ there exist equivalences $$\begin{tikzcd} \psi_k : D^b(X_+) \overset{\sim}{\To} D^b(X_-). \end{tikzcd}$$\end{keythm}
The fact that there are many different choices of windows is not a surprise, as it was present in the original analysis of Herbst--Hori--Page. It has an important consequence: if we combine equivalences corresponding to different windows, we produce \emph{autoequivalences} of $D^b(X_+)$.

\begin{keydefn}[Definition \ref{defn.window-autoequivalences2}] We define {\em window-shift autoequivalences} $\omega_{k,l}$ by
$$\omega_{k,l} := \psi_k^{-1}\psi_l : D^b(X_+) \overset{\sim}{\To} D^b(X_+).$$
\end{keydefn}

Most of this paper is devoted to studying these autoequivalences, and in particular to proving that they are equivalently described by a geometric construction discovered by the first author in \cite{Donovan:2011vc}. In the case of a standard 3-fold flop this geometric construction is well-known -- the skyscraper sheaf along the flopping $\P^1$ is a spherical object, and we can get a derived autoequivalence by performing a Seidel-Thomas spherical twist \cite{Seidel:2000}. In the Grassmannian examples, the construction produces something a bit like a family spherical twist \cite{Horja:2001} but more complicated: in fact it is associated to a spherical functor \cite{Anno:2007wo} which involves a push-down by a resolution of singularities. See \cite{Donovan:2011vc} for further discussion in the case $r=2$. We show:

\begin{keythm}[Theorem \ref{theorem.grassmantwist}] There exists a natural isomorphism $$\omega_{0,1} \iso T_F$$ where $T_F$ is a twist of a spherical functor $F$ with target $D^b(X_+)$, defined in Section \ref{section.geometry}.
\end{keythm}

Dually, we find that another window shift autoequivalence, namely $\omega_{-1,0}$, can be described in terms of a cotwist \cite{Anno:2010we} around a spherical functor with source $D^b(X_+)$: we defer a precise statement until Section \ref{section.geometry} (Theorem \ref{theorem.grassmancotwist}). As a pleasing corollary, we find that a twist and a cotwist on $D^b(X_+)$ are related (Corollary \ref{corollary.twistcotwist}).
\medskip

The physics in \cite{Herbst:2008wl} concerns B-branes in gauged linear $\sigma$-models (GLSMs). The input data for such a model consists of a vector space $M$ with an action of a group $G$, then by standard prescriptions one can build a supersymmetric gauge theory in 2 dimensions. The theory has a complex parameter $t$, called the Fayet--Iliopoulos parameter, and in certain \quotes{large-radius} limits this gauge theory reduces to a non-linear $\sigma$-model with target space given by a GIT quotient $M \quot G$: different quotients appear at different limits. The parameter $t$ becomes identified, in the limit, with the (complexified) K\"ahler class of the target space, so the space in which $t$ lives is called the Stringy K\"ahler Moduli Space (SKMS).

The B-branes in the theory form a category, which in the limit is the derived category of $M \quot G$. Furthermore when $G\subset \operatorname{SL}(M)$ this category is actually independent of $t$, so all the GIT quotients are derived equivalent. However, to produce a derived equivalence one must vary $t$ from one large-radius limit point to a different one, and in between the description of the B-branes as the derived category of a space breaks down. Herbst--Hori--Page instead study the B-branes at a different kind of limit, the \quotes{Coloumb phase} of the theory, and in doing so discover  \quotes{grade-restriction rules}, which we choose to call \quotes{windows}.

 The Coulomb phase description arises when $t$ is near certain singularities in the SKMS. Because of these singularities, when we move from one large-radius limit to another there are many homotopy classes of paths that we can choose to move along, which is  why there are many different choices of windows with different corresponding equivalences. In this picture, we see our autoequivalences as coming from monodromy of B-branes as $t$ moves along 
loops around the singularities. 

 Herbst--Hori--Page restrict to the case that $G$ is a torus, whereas in our class of examples we consider the non-abelian gauge group $U(r)$. GLSMs with non-abelian gauge groups have certainly been studied \cite{Hori:2006ug}, so we hope that our calculations of brane monodromy in these theories will be of interest to some physicists.\medskip

The plan of this paper is as follows:

\begin{itemize}
\item Section \ref{section.examples_and_heuristics} is intended to give a readable introduction to our methods, without the morass of Schur functors that arises in the general case. We describe the case of the standard flop in some detail, and provide some discussion of the simplest Grassmannian example.
\item In Section \ref{section.proofs} we give precise descriptions of all the algebraic and geometric constructions, and give the proofs that they are equivalent.
\item In Appendix \ref{appendix} we prove various technical results that are required. In particular we make extensive use of some long exact sequences on Grassmannians, and since these are non-standard we give an explicit description of them.
\end{itemize}

Whilst this paper was being completed, we learnt of the work of Halpern-Leistner \cite{HalpernLeistner:2012uha} and Ballard--Favero--Katzarkov \cite{Ballard:2012wia}, both of which apply related methods to produce derived equivalences corresponding to very general variations of GIT. We hope to treat the associated window-shift autoequivalences in future work.

\begin{acks}We would like to thank Nick Addington, Timothy Logvinenko and Richard Thomas for helpful conversations and suggestions, and also Debbie Levene for a useful observation on the combinatorics.

W.D. would like to express his gratitude for the support of Iain Gordon and EPSRC grant EP/G007632. E.S. is grateful for the support of an Imperial College Junior Research Fellowship.
\end{acks}
\section{Examples and heuristics}
\label{section.examples_and_heuristics}
\begin{notn}When discussing derived categories, functors are derived unless stated otherwise. Curly braces denote a complex of sheaves understood as an object of a derived category: an underline records the position of the degree $0$ term.
\end{notn}
\subsection{Windows and window-shifts}
\label{section.intro_window_shifts}
\subsubsection{The standard 3-fold flop}
\label{section.intro_standard_flop}

We'll start by considering the example of the standard 3-fold flop. We let $V$ be a 2-dimensional vector space over $\C$, and we let $\C^*$ act on $V$ via the vector space structure. This induces an action
$$\C^* \acts  V \oplus V^\vee.$$
We consider the two possible  GIT quotients under this action. For the first one we throw away the subspace $\{0\} \oplus V^\vee$ and get a quotient
 $$X_+ = \Tot \left(\cO(-1)^{\oplus 2}_{\P V}\right).$$
For the second one we throw away $V \oplus \{0\}$ and get
$$X_- = \Tot \left(\cO(-1)^{\oplus 2}_{\P V^\vee}\right).$$
So both $X_+$ and $X_-$ are non-compact Calabi-Yau 3-folds, and they are birational (they also happen to be isomorphic). It's well-known \cite[Theorem 3.6]{Bondal:1995vra} that $X_+$ and $X_-$ are also derived equivalent.

A particular way of viewing this derived equivalence was introduced by the second author in \cite{Segal:2009tua}, based on the work of Herbst, Hori and Page \cite{Herbst:2008wl}. What we do is view $X_+$ and $X_-$ as open substacks of the Artin stack
$$\mfX = [V\oplus V^\vee \;/\; \C^*]$$ 
and write $i_{X_+}$ and $i_{X_-}$ for the respective inclusions, as illustrated in Figure \ref{figure.quotients}. 
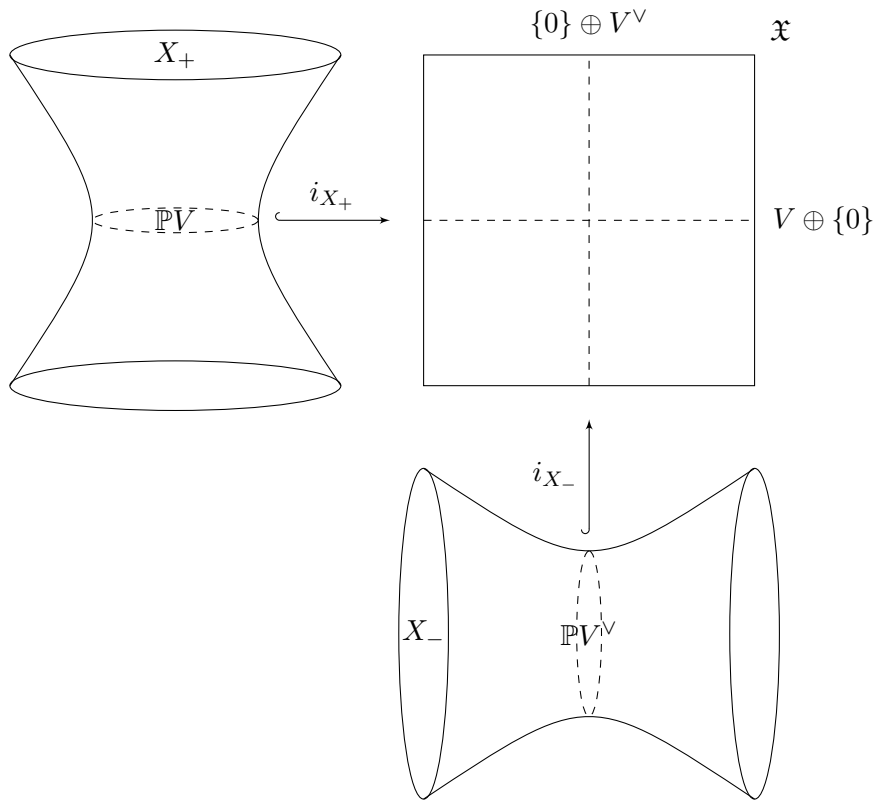
\begin{figure}[h]
\begin{center}
\begin{tikzpicture}[node distance=2.5cm,auto,>=latex',scale=1.1]

\def\quotHeight{4}
\def\quotHalfWidth{2}
\def\quotPerspective{0.15}
\def\quotZeroSectionScale{0.5}
\def\quotPinch{0.67}

\draw (3,0) rectangle (7,4)
  node (bigV) [above right=2] {$\mfX$} ; 
\draw[dashed]
(5,0)
  node (below V*) [below=5] {}
--
(5,4)
  node (V*) [above=2] {$\{0\} \oplus V^\vee$}; 

\draw[dashed]
(3,2)
  node (left of V) [left=5] {}
  node (left on V) [below right=3] {} 
--
(7,2)
  node (V) [right=3] {$V \oplus \{0\} $}
  node (right on V) [left=10] {}; 


\foreach
  \x/\y/\rotation/\name/\projection/\zeroSection/\zeroSectionLabel/\ZSNodeName/\sideZSNodeName/\sideDirection
  in
  {0/0/0/X_+/\pi/\PP V//PV/right of PV/right,
  -3/-7/90/X_-//\PP V^\vee//PV*/above PV*/above}
{
\pgftransformrotate{\rotation}
\draw (\x,\y) ellipse ({\quotHalfWidth} and {\quotHalfWidth * \quotPerspective});
\draw (\x,\y + \quotHeight) ellipse ({\quotHalfWidth} and {\quotHalfWidth * \quotPerspective}) node (X)  {$\name$};
\draw[dashed] (\x,{\y + \quotHeight / 2}) ellipse ({\quotZeroSectionScale * \quotHalfWidth} and {\quotZeroSectionScale * \quotHalfWidth * \quotPerspective})
  node (\ZSNodeName) {$\zeroSection$}
  node [\sideDirection=30] (\sideZSNodeName) {}
  node [below=10] {$\zeroSectionLabel$};

\draw ({\x - \quotHalfWidth},\y) .. controls ({\x - \quotPinch},{\y + \quotHeight / 2}) .. ({\x - \quotHalfWidth},{\y + \quotHeight});
\draw ({\x + \quotHalfWidth},\y) .. controls ({\x + \quotPinch},{\y + \quotHeight / 2}) .. ({\x + \quotHalfWidth},{\y + \quotHeight});

}



\draw [right hook->] (right of PV) to node {$i_{X_+}$} (left of V);
\draw [right hook->] (above PV*) to node {$i_{X_-}$} (below V*);

\end{tikzpicture}
\end{center}
\caption{Notation for GIT quotients $X_{\pm}$, viewed as substacks of the stack $\mfX$.}
\label{figure.quotients}
\end{figure}

On $\mfX$ we have a tautological line bundle $\cO(1)$.  We want to consider the subcategory
$$\cW_0 := \left< \cO, \cO(1) \right> \subset D^b(\mfX) $$
which is by definition split-generated by the trivial and tautological line bundles. We call this subcategory a \textit{window}. Its significance is the following:
\begin{prop}\label{prop.restrictionequivalences1} Both functors
$$i_{X_\pm}^* : \cW_0 \To D^b(X_\pm)$$
are equivalences.
\end{prop}
This proposition is easy to prove: it follows rapidly (see Proposition \ref{prop.restrictionequivalences2}) from the statement that the bundle $\cO\oplus\cO(1)$ is tilting on both $X_+$ and $X_-$. This is deduced from Beilinson's theorem \cite{Beilinson:2009tf}, which says that $\cO$ and $\cO(1)$ form a full strong exceptional collection on $\P^1$. Hence we have a derived equivalence
$$\psi_0 : D^b(X_+) \overset{\sim}{\To} D^b(X_-)$$
defined as the composition
$$\begin{tikzcd}[column sep=40pt] D^b(X_+) \rar{\sim}[swap]{( i_{X_+}^*)^{-1}} & \cW_0 \rar{\sim}[swap]{i_{X_-}^*} & D^b(X_-). \end{tikzcd}$$

We can calculate the effect of this equivalence quite explicitly. Take a sheaf (or complex) $E\in D^b(X_+)$. Resolve $E$ by the bundles $\cO$ and $\cO(1)$, then this determines an extension of $E$ to an object $\cE \in\cW_0\subset D^b(\mfX)$. Now we can restrict $\cE$ to get an object in $D^b(X_-)$.

This gets more interesting when we notice that $\cW_0$ is not the only window that we could have chosen. Indeed for any $k\in\Z$ we can define
$$\cW_k := \left< \cO(k), \cO(k+1) \right> \subset D^b(\mfX) $$
and Proposition \ref{prop.restrictionequivalences1} will hold for $\cW_k$. So we have a whole set of derived equivalences $\set{\psi_k}$, according to which window we choose to pass through, and it turns out they are all distinct. If we combine them, we can produce autoequivalences
$$\omega_{k,l} := \psi_k^{-1}\psi_l: D^b(X_+) \overset{\sim}{\To} D^b(X_+).$$
We call these {\em window-shift} autoequivalences. Of course they're not independent, rather they obey the following relations:
\begin{equation}\label{equation.windowshiftrelations}\begin{gathered}\omega_{m,k}\omega_{k,l} = \omega_{m,l} \\ \omega_{k+m,l+m} = (\otimes \cO(m))\circ \omega_{k,l} \circ (\otimes \cO(-m)) \end{gathered}\end{equation}
Window-shifts can be calculated explicitly, at least in principle. As an example, let's calculate the effect of the window-shift $\omega_{-1,0}$ on the two line bundles $\cO$ and $\cO(1)$. Applying the first functor $\psi_0$ is easy, these two bundles immediately lift to $\cW_0$ so we have
$$\psi_0(\cO) = \cO, \quad\quad \psi_0(\cO(1)) = \cO(1),$$
in $D^b(X_-)$. We're adopting a particular sign convention here: since $\C^*$ is acting with weight $-1$ on $V^\vee$, it seems reasonable to declare that on $\P V^\vee$ it is the $\cO(-1)$ line bundle that has global sections, not the $\cO(1)$ line bundle. If we weren't using this convention then we'd have $\psi_0(\cO(1)) = \cO(-1)$.

To apply the second functor $\psi_{-1}^{-1}$ we have to resolve $\cO(1)$ in terms of $\cO(-1)$ and $\cO$, so that we can move back through the window $\cW_{-1}$. On $X_-$ we have an exact sequence given by
\begin{equation}\label{eqn.longeuler} \begin{tikzcd} 0 \rar & \cO(1)\otimes \det(V) \rar & \cO\otimes V \rar & \cO(-1) \rar & 0 \end{tikzcd} \end{equation}
which is the pull-up of the Euler sequence on $\P V^\vee$. Consequently, after picking a basis for $V$ we have
\begin{equation}\label{eqn.resultsofwindowshift}\omega_{-1,0}(\cO) = \cO, \quad\quad \omega_{-1,0}(\cO(1)) = \complex{ \underline{\cO^{\oplus 2}} \longrightarrow \cO(-1) }.\end{equation}
It is straightforward, but more fiddly, to compute the effect of $\omega_{-1,0}$ on $\cO(k)$ for other $k$.

\subsubsection{Grassmannian flops}
\label{section.intro_grassmannian}

The strategy given in Section \ref{section.intro_standard_flop} should lead to derived equivalences, and autoequivalences, in many more examples. In this paper, we're only going to generalize in the following way. Let $V$ now be a vector space of arbitrary dimension $d$. Also, let $S$ be another vector space with dimension $r$, where $r\leq d$. We form the Artin stack
$$\mfX^{(d,r)} = \left[\Hom(S,V)\oplus \Hom(V,S) \; /\; \operatorname{GL}(S) \right].$$
We then have two possible GIT quotients given by open substacks $X_\pm^{(d,r)}$ of $\mfX^{(d,r)}$. It is straightforward to establish (cf. \cite[Proposition 4.14]{Rica}) that one quotient $X_+^{(d,r)}$ is the locus where the map from $S$ to $V$ is full rank: it's the total space of a vector bundle over the Grassmannian $\Gr(r,V)$. Similarly $X_-^{(d,r)}$ is the total space of a vector bundle over the dual Grassmannian $\Gr(V, r)$. Note that setting $d=2$ and $r=1$ recovers the 3-fold flop. As before $X_\pm^{(d,r)}$ are non-compact Calabi-Yau \cite[Section 3.2]{Donovan:2011ua}.

To apply our strategy we first need to know a (full strong) exceptional collection on $\Gr(r,V)$: such a collection was discovered by Kapranov \cite{Kapranov:2009wf}. It consists of particular Schur powers of the tautological bundle $S$, for example in the case $d=4$, $r=2$ the exceptional collection is 
$$\set{ \cO,\, S^\vee, \,\Sym^2 S^\vee, \,\cO(1),\, S^\vee(1), \, \cO(2)} $$
where $\cO(1)=\det S^\vee$. Now we can define our windows: this same set of Schur powers determines a set of bundles on the stack $\mfX^{(d,r)}$, and we let 
$$\cW_0 \subset D^b \left(\mfX^{(d,r)}\right)$$
be the subcategory that they split-generate. To get the other windows $\cW_k$ we tensor every bundle in the collection by $\cO(k)$. The analogue of Proposition \ref{prop.restrictionequivalences1} still holds (see Proposition \ref{prop.restrictionequivalences2}), so we get equivalences
$$\psi_k:  D^b \left(X_+^{(d,r)}\right) \overset{\sim}{\To} D^b \left(X_-^{(d,r)}\right)$$
by passing through each window $\cW_k$, and  combining them we get window-shift autoequivalences
 $$\omega_{k,l}:=\psi_k^{-1}\psi_l: D^b \left(X_+^{(d,r)}\right) \overset{\sim}{\To} D^b \left(X_+^{(d,r)}\right).$$

With very little work, we have produced some novel derived autoequivalences. However the method is very algebraic, and it would be nice to have some geometric understanding of them. This is a much harder question, which we will turn to in the next section.
\medskip

Before we do that, let's present one more explicit calculation. Hopefully this will give the reader some feel for the computations that are going to arise later on in the paper. Let's look at the effect of the window-shift $\omega_{-1,0}$, as we did before, but this time let's do it in the case $d=4$, $r=2$. 
  As before let's make life easy by only looking at the effect of $\omega_{-1,0}$ on the generating bundles for $\cW_0$. Then we have immediately that
 \[\psi_0 (\cO) = \cO, \quad\quad \psi_0 (S^\vee) = S^\vee,  \quad\quad\ldots, \quad\quad\psi_0(\cO(2)) = \cO(2).\] 
The bundles $\cO$, $S^\vee$ and $\cO(1)$ also lie in the generating set for the window $\cW_{-1}$, so applying $\psi_{-1}^{-1}$ to them is easy, and we have 
\[\omega_{-1,0} (\cO) = \cO, \quad\quad \omega_{-1,0} (S^\vee) = S^\vee, \quad\quad \omega_{-1,0} (\cO(1)) = \cO(1).\] 
Obviously this is a general phenomenon: $\omega_{k,l}$ fixes any bundles that lie in the generating sets for both $\cW_k$ and $\cW_l$.  

Now let's calculate the effect of $\omega_{-1,0}$ on $\Sym^2 S^\vee$. To apply $\psi_{-1}^{-1}$ we have to resolve $\Sym^2 S^\vee$ in terms of the window $\cW_{-1}$. It turns out that there is an exact sequence on $X_{-}^{(4,2)}$ given by
\begin{equation}\label{eqn.G24complex1} \begin{tikzcd}[column sep=25pt] 0 \rar & \Sym^2 S^\vee \otimes \wedge^4 V \rar & S^\vee \otimes \wedge^3 V \rar & \cO \otimes \wedge^2 V \rar & \cO(-1) \rar & 0 \end{tikzcd} \end{equation}
which is the pull-up from $\Gr(4,2)$ of (a twist of) an Eagon-Northcott complex \cite{Eagon:1962ua} (see Example \ref{eg.eagon-northcott}). Hence, after picking a basis for $V$ again, we have:
 \[\omega_{-1,0}(\Sym^2 S^\vee) =  \complex{  \underline{S^{\vee \oplus 4}} \longrightarrow  \cO^{\oplus 6} \longrightarrow  \cO(-1) }.\]
To calculate $\omega_{-1,0}(S^\vee(1))$ we use the exact sequence 
\begin{equation}\label{eqn.G24complex2} \begin{tikzcd} 0 \rar & S^\vee (1) \otimes \wedge^4 V \rar & \cO(1) \otimes \wedge^3 V \rar & \cO \otimes V \rar & S^\vee (-1) \rar & 0 \end{tikzcd} \end{equation}
which is the pull-up from  $\Gr(4,2)$ of (a twist of) a Buchsbaum-Rim complex \cite{Buchsbaum:1964uq} (see Example \ref{eg.buchsbaum-rim}). Then
\[\omega_{-1,0}(S^\vee (1)) = \complex{  \underline{\cO(1)^{\oplus 4}} \longrightarrow  \cO^{\oplus 4} \longrightarrow  S^\vee(-1)}.\]
The calculation for $\omega_{-1,0}(\cO(2))$ requires a third sort of \quotes{generalised Koszul complex}: it's the complex denoted $\mathcal{C}^2$ in \cite[Appendix A.2]{Eisenbud:1994vv}. Pulling it up to $X_-^{(4,2)}$ and twisting we get
\begin{equation} \label{eqn.G24complex3} \begin{tikzcd}[column sep=22pt] 0 \rar & \cO(2) \otimes \wedge^4 V \rar & \cO(1) \otimes \wedge^3 V \rar & S^\vee \otimes V \rar & \Sym^2 S^\vee (-1) \rar & 0 \end{tikzcd} \end{equation}
so
\[\omega_{-1,0}(\cO(2)) = \complex{ \underline{\cO(1)^{\oplus 4}} \longrightarrow  S^{\vee \oplus 4} \longrightarrow  \Sym^2 S^\vee(-1)}.\]
 Evidently to do these calculations in general we'd need to know a lot of exact sequences on Grassmannians. In fact for $r=2$ the complexes $\mathcal{C}^i$ in {\it loc. cit.} suffice, but for higher $r$ we need generalizations. We'll return to this point later.

\begin{figure}[H]

\newcommand\triangleNodes[9]{
\def\factor{#1}
\def\xshift{#2}
\node ({#3}Schur22) at (-4 * \factor + \xshift, 0) {#4};
\node (above{#3}Schur22) at (-4 * \factor + \xshift, 1) {};
\node ({#3}Schur11) at (-2 * \factor + \xshift, 0) {#5};
\node ({#3}Schur00) at (0 + \xshift,0) {#6};
\node (above{#3}Schur00) at (0 + \xshift, 1) {};
\node ({#3}Schur21) at (-3 * \factor + \xshift, 1 * \factor) {#7};
\node ({#3}Schur10) at (-1 * \factor + \xshift, 1 * \factor) {#8};
\node ({#3}Schur20) at (-2 * \factor + \xshift, 2 * \factor) {#9};
\draw[gray] (0 + \xshift,0) -- (-4 * \factor + \xshift, 0) -- (-2 * \factor + \xshift, 2 * \factor) -- cycle;
}

\newcommand\twoTriangleNodes[6]{
\def\factor{#1}
\def\xshift{#2}
\node ({#3}Schur33) at (-6 * \factor + \xshift, 0) {#4};
\node ({#3}Schur32) at (-5 * \factor + \xshift, 1 * \factor) {#5};
\node ({#3}Schur31) at (-4 * \factor + \xshift, 2 * \factor) {#6};
\draw[gray] (-2 * \factor + \xshift, 0) -- (-6 * \factor + \xshift, 0) -- (-4 * \factor + \xshift, 2 * \factor) -- cycle;
}

\newcommand\twoWindowNodes[1]{
\triangleNodes{#1}{0}{window}{$\cO(1)$}{$\cO$}{$\cO(-1)$}{$S^\vee$}{$S^\vee(-1)$}{$\Sym^2 S^\vee(-1)$}
\twoTriangleNodes{#1}{0}{window}{$\cO(2)$}{$S^\vee(1)$}{$\Sym^2 S^\vee$}
}

\begin{center}
\begin{tikzpicture}
\twoWindowNodes{1}

\node (W0) at (-7,2) {$\cW_0$};
\node (W-1) at (1,2) {$\cW_{-1}$};
\draw [thick,->] (W0) -- ({window}Schur32);
\draw [thick,->] (W-1) -- ({window}Schur10);

\end{tikzpicture}
\end{center}
\caption{Windows used in calculation of window-shift $\omega_{-1,0}$ for Grassmannian example $d=4$, $r=2$.}
\end{figure}

\subsection{Spherical twists}\label{section.intro_sphericaltwists}

Let's return to the example of the 3-fold flop. We have our 3-fold $X_+ = X_+^{(2,1)}$, and we may consider the window-shift autoequivalence
$$\omega_{0,1}: \begin{tikzcd}D^b(X_+) \rar{\sim} & \cW_1 \rar{\sim} & D^b(X_-) \rar{\sim} & \cW_0 \rar{\sim} & D^b(X_+).\end{tikzcd}$$
Observe that the zero section $\P V$ inside $X_+$ is precisely the locus that becomes unstable when we pass to the other GIT quotient $X_-$.  Away from $\P V$ the two quotients are isomorphic, and the equivalences $\psi_k$ are just the identity, so the effect of the window-shift is concentrated along $\P V$. It was argued (somewhat imprecisely) in \cite{Segal:2009tua} that $\omega_{0,1}$ is in fact a Seidel-Thomas spherical twist \cite{Seidel:2000} around the spherical object
$$\cO_{\P V} \in D^b(X_+).$$
This result was already folklore, at least in the physics literature. To define this spherical twist, we consider $\P V$ as a correspondence:
\begin{equation}\label{eqn.basiccorrespondence1}\begin{tikzcd}  & \P V \drar{j}  \dlar[swap]{\pi} & \\  \pt & & X_+\end{tikzcd} \end{equation}
Then we have a functor
$$F =j_*\pi^*:  D^b(\pt) \To D^b(X_+), $$
and its right adjoint
$$R= \pi_*j^!: D^b(X_+) \To D^b(\pt).$$
The adjunction gives a natural transformation
$$ j_*\pi^*\pi_* j^! \To \id,$$
and the \textit{spherical twist} 
$$T_{F} : D^b(X_+) \To D^b(X_+)$$
is the cone on this natural transformation. It's immediate that
$$T_{F}(E) = \cone{  \Hom(\cO_{\P V}, E)\otimes \cO_{\P V} \To E }$$
which is perhaps a more standard definition (but of course we are anticipating a generalization). 
\begin{prop}\label{prop.shiftequalstwist1}
The window-shift $\omega_{0,1}$ and the spherical twist $T_{F}$ coincide.
\end{prop}
This is a special case of our later Theorem \ref{theorem.grassmantwist}, but we'll sketch the proof here. Suppose we wanted to compute the effect of the window-shift $\omega_{0,1}$ on some object $E\in D^b(X_+)$. Firstly we resolve $E$ by the bundles in $\cW_1$, then we can apply $\psi_1$ and get an object $\psi_1 E \in D^b(X_-)$. Secondly we need to rewrite $\psi_1 E$ in terms of the other window $\cW_0$, then we can apply $\psi_0^{-1}$ and bring it back to $D^b(X_+)$. The key idea of the proof is to find an endofunctor
$$T_{\cF}:  D^b(\mfX) \To D^b(\mfX)  $$
on the stack $\mfX$ that carries out this second step of the window-shift, i.e. it rewrites objects from $\cW_1$ in terms of $\cW_0$. Then we need to know that on $X_+$ the functor $T_{\cF}$ acts as  the spherical twist. Specifically, we want a functor that has the following three properties:
\begin{enumerate}
\item The effect of $T_{\cF}$ is concentrated along the locus $V\oplus \set{0}$, so it acts as the identity on $X_-$. More precisely, we want
 $$i^*_{X_-} T_{\cF} = i^*_{X_-}.$$
In fact it is enough that this equality holds on the subcategory $\cW_1$.
\item $T_{\cF}$ maps the window $\cW_1$ to the window $\cW_0$. 
\item When we restrict $T_{\cF}$ to $X_+$ it acts as the spherical twist $T_F$, i.e. the diagram
 $$\begin{tikzcd}
D^b(\mfX) \rar{T_\cF} \dar[swap]{i_{X_+}^*} & D^b(\mfX) \dar{i_{X_+}^*} \\ 
 D^b(X_+) \rar{T_F} & D^b(X_+) 
\end{tikzcd}$$
commutes. Again, it's actually enough that the diagram commutes when we restrict to the subcategory $\cW_1$. 
\end{enumerate}
We call a functor with these properties a \textit{transfer functor}, since it transfers between windows.
\medskip

If we have a transfer functor $T_\cF$, then the proof of Proposition \ref{prop.shiftequalstwist1} is an immediate formality:

\begin{proof}[Proof of Proposition \ref{prop.shiftequalstwist1}] 
Using property (ii), we have a diagram
\begin{center}
\begin{tikzcd}[column sep=35pt, row sep=35pt]
   \cW_1 \ar{rr}{T_\cF} \ar{dd}[swap]{i_{X_+}^*}[above,sloped]{\sim} \ar{dr}{i_{X_-}^*}[below,sloped]{\sim} & & \cW_{0} \ar{dd}{i_{X_+}^*}[below,sloped]{\sim} \ar{dl}[swap]{i_{X_-}^*}[below,sloped]{\sim} \\
  & D^b(X_-) & \\ 
  D^b(X_+) \ar{rr}{T_F} \ar{ur}{\psi_1}[below,sloped]{\sim} & & D^b(X_+) \ar{ul}[swap]{\psi_0}[below,sloped]{\sim}
\end{tikzcd}
\end{center}
The left- and right-hand triangles commute by definition, and the top triangle and the outer square both commute by properties (i) and (iii). Noting that the left-hand side of the outer square is an isomorphism, we then see that the bottom triangle commutes.
\end{proof}

 It's not difficult to guess what the transfer functor $T_\cF$ is: it's the exact analogue of $T_F$ for the stack $\mfX$. We consider the correspondence
\begin{equation}\label{eqn.basiccorrespondencestack}\begin{tikzcd} & {[V\,/\,\C^*]} \drar{j}  \dlar[swap]{\pi} & \\  \pt & & \mfX\end{tikzcd} \end{equation}
then $T_\cF$ is the cone
$$\cone{ j_*\pi^*\pi_* j^! \To \id } : D^b(\mfX) \To D^b(\mfX).$$
Now it's just a matter of checking the properties (i)--(iii), but as this gets rather involved in the general case it's probably worth saying a few words about it here. 
\begin{enumerate}
\item This property is obvious from the definition.
\item We need to calculate $T_\cF(\cO(1))$ and $T_\cF(\cO(2))$ and check that they both end up in $\cW_0$. Firstly, the relative canonical bundle $K_j$ of $j$ is $\cO(-2)$, and its relative dimension is $-2$. So
$$j^!(\cO(1)) = K_j\otimes j^*(\cO(1))[\dim j] = \cO(-1) [-2].$$
Hence $\pi_*j^!(\cO(1)) = 0$, but then $T_\cF(\cO(1)) = \cO(1)$, and this is indeed in $\cW_0$. 

The calculation of $T_\cF(\cO(2))$ is a little more complicated. We have
$$j^!(\cO(2)) = \cO[-2]$$
so $\pi_*j^!(\cO(2)) = \cO_{\pt}[-2]$, and
$$j_*\pi^*\pi_*j^!(\cO(2)) = \cO_V[-2]  \in D^b(\mfX).$$ 
We know from the adjunction that there is supposed to be a natural map
\begin{equation} \label{eqn.adjinbasiccase} \cO_V[-2] \To \cO(2). \end{equation}
To see it explicitly, we need to use the Koszul resolution of $\cO_V$ given by
\begin{equation}\label{eqn.basicfreeresolution1}\begin{tikzcd} 0 \rar & \cO(2) \rar & \cO(1)^{\oplus 2} \rar & \cO \rar & \cO_V \rar & 0. \end{tikzcd}\end{equation}
When we take the cone on \eqref{eqn.adjinbasiccase} the two copies of $\cO(2)$ cancel out, and the result is quasi-isomorphic to the complex
$$\complex{\underline{\cO(1)^{\oplus 2}} \To \cO}.$$
This is in $\cW_0$, as required.

\item This property is not surprising given that the definitions of $T_F$ and $T_\cF$ are so closely related, but there is something to check. The issue is that calculating $\pi_*$ from the space $\P V$ can give a different answer than if we calculate it from the stack $[V\,/\,\C^*]$, because on $\P V$ sheaves can have higher cohomology. However the two bundles $\cO(-1)$ and $\cO$ have no higher cohomology, which means that the functors  $T_F i_{X_+}^*$ and $i_{X_+}^* T_\cF$ give the same results when we restrict to the window $\cW_1$. They are not however the same on the whole of $D^b(\mfX)$.

\end{enumerate}

\subsection{Spherical cotwists}\label{section.intro_sphericalcotwists}

Now let's look for a geometric interpretation for the window-shift autoequivalence
$$\omega_{-1, 0}: D^b(X_+) \overset{\sim}{\To} D^b(X_+).$$
Up to tensoring with $\cO(1)$ this is inverse to $\omega_{0,1}$, by the relations \eqref{equation.windowshiftrelations}. The relevant geometrical functor is an example of an (inverse) \emph{spherical cotwist}.

We need to use the natural map
$$X_+ \To \Hom(V,V) \iso \C^4$$
which contracts the zero section and  has a 3-fold ordinary double point $\Im(X_+)$ as its image. To maintain symmetry with Section \ref{section.intro_sphericaltwists} (and the general case which we'll meet later), we're going to write this as a correspondence as follows:
\begin{equation}\label{eqn.basiccorrespondence2}\begin{tikzcd}  & X_+ \drar{j}  \dlar[swap]{\id} &[-15pt] \\  X_+ & & \Hom(V,V)\end{tikzcd} \end{equation}
Then analogously we have a  functor
$$F= j_* :  D^b(X_+) \To D^b(\Hom(V,V)).$$
This time we're going to use its {\em left} adjoint, which is 
$$L= j^*:  D^b(\Hom(V,V))\To D^b(X_+),$$
and form the cone
$$\cone{ j^*j_* \To \id}:  D^b(X_+) \To D^b(X_+).$$
\begin{prop}\label{prop.shiftequalscotwist1}
The window-shift $\omega_{-1,0}$ is equal to the shifted cone
$$\cone{j^*j_* \To \id}\,[-2].$$
\end{prop}
The proof of this proposition follows exactly the same structure as the proof of Proposition \ref{prop.shiftequalstwist1}, i.e. we find a transfer functor on $D^b(\mfX)$ which restricts to the given functor on $D^b(X_+)$. However, describing this transfer functor would require us to go into more detail than we wish to at this point, so for the moment we'll just do a heuristic calculation (for the full proof see Theorem \ref{theorem.grassmancotwist}). What we'll do is  show that these two functors give the same answer on the bundles $\cO$ and $\cO(1)$. These bundles generate all of $D^b(X_+)$, so this is some evidence that the functors are the same. We'll also see some of the kind of computations that will be needed in the proof of the general case.

We've already calculated the effect of $\omega_{-1,0}$ on $\cO$ and $\cO(1)$: the answer is given in \eqref{eqn.resultsofwindowshift}. So let's calculate the functor $[j^*j_*\to \id]$ on these two bundles and compare. We have
$$j_*(\cO)=\cO_{\Im(X_+)}$$
on $\Hom(V,V)$. To apply $j^*$, we need to know that this has a free resolution
\begin{equation}\label{eqn.basicfreeresolution2}\begin{tikzcd}0 \rar & \cO \rar & \cO \rar & \cO_{\Im(X_+)} \rar & 0. \end{tikzcd}\end{equation}
Taking the cone to the identity kills the second copy of $\cO$, so
$$ \cone{j^*j_* \To \id} : \cO \longmapsto \cO[2],$$
and hence the shifted cone agrees with the window-shift on the bundle $\cO$. To do the calculation for $\cO(1)$, we observe that $j_*(\cO(1))$ is a sheaf supported on $\Im(X_+)$, and it has a free resolution
\begin{equation}\label{eqn.basicfreeresolution3}\begin{tikzcd}0 \rar & \cO^{\oplus 2} \rar & \cO^{\oplus 2} \rar & j_*(\cO(1))\rar & 0 \end{tikzcd}\end{equation}
so $ [j^*j_* \to \id ]$ maps $\cO(1)$ to
$$\complex{ \cO^{\oplus 2} \To \cO^{\oplus 2} \To \underline{\cO(1)}  } $$ 
which is quasi-isomorphic to 
$$\complex{ \cO^{\oplus 2} \To \underline{\cO(-1)}  } [1],$$
using the analogue of the exact sequence \eqref{eqn.longeuler} on $X_+$. After shifting by $[-2]$, this agrees with $\omega_{-1,0}$.

\subsection{Grassmannian twists}
\subsubsection{Statement of results}

Now let's turn to the general case (at least for us!), namely the Grassmannian flop
$$\begin{tikzcd}[column sep=19pt] X_{+\phantom{-}}^{(d,r)} \rar[dashed] & \lar[dashed] X_{-\phantom{+}}^{(d,r)}.\end{tikzcd} $$
We seek a geometric interpretation for our window-shift autoequivalences
$$\omega_{k,l}: D^b \left(X_+^{(d,r)}\right) \overset{\sim}{\To} D^b \left(X_+^{(d,r)}\right).$$
As before, the effect of the window-shift is concentrated on the flopping locus where the birational map fails to be an isomorphism, i.e. the locus in $X_+^{(d,r)}$ that becomes unstable when we pass to the other GIT quotient. In the 3-fold flop case $X_+^{(2,1)}$ this locus was just the zero section $\P V$, but in a general $X_+^{(d,r)}$ it's much more complicated, and in particular is usually non-compact. Nevertheless, the geometric constructions of Sections \ref{section.intro_sphericaltwists} and \ref{section.intro_sphericalcotwists} can be generalized. 

Consider the correspondences \eqref{eqn.basiccorrespondence1} and  \eqref{eqn.basiccorrespondence2} that we used for $X_+^{(2,1)}$. The key point to notice is that $\Hom(V,V)$ is actually $X_+^{(2,2)}$, and that $X_+^{(2,0)}$ is a point! So in general we should be looking for correspondences as follows:
\begin{equation}\label{eqn.chainofcorrespondences}
\begin{tikzcd}[column sep=12pt] \ldots \drar{j}&  & Z^{(d,r-1, r)}\dlar[swap]{\pi} \drar{j}& & Z^{(d,r, r+1)} \dlar[swap]{\pi} \drar{j}& & \dlar[swap]{\pi} \ldots \\
                 & X_+^{(d,r-1)} & & X_+^{(d,r)} & & X_+^{(d,r+1)} & \end{tikzcd}
\end{equation}
The relevant correspondences for the $r=2$ case were described by the first author in \cite{Donovan:2011ua}: we describe the general case in Section \ref{section.grassmannians}. Then we have functors
$$F = j_*\pi^*: D^b \left(X_+^{(d,r)}\right) \To D^b\left(X_+^{(d, r+1)}\right) $$
which have right and left adjoints $R$ and $L$, and we form the \textit{twist} functors
$$T_F := \cone{ FR \To \id } :  D^b\left(X_+^{(d,r)}\right) \To D^b\left(X_+^{(d, r)}\right) $$
and \textit{inverse cotwist} functors
$$C_F^{-1}:= \cone{ LF \To \id }[-1]: D^b\left(X_+^{(d,r)}\right) \To D^b\left(X_+^{(d, r)}\right). $$
We then prove (Theorems \ref{theorem.grassmantwist} and  \ref{theorem.grassmancotwist}) that the twist functor $T_F$ is equal to the window-shift $\omega_{0,1}$, and the inverse cotwist functor $C_F^{-1}$ is equal to the window-shift $\omega_{-1,0}$ (up to a shift in homological degree). 

\subsubsection{Remarks on the proofs}

The structure of our proofs remains the same as in the 3-fold flop example: we find transfer functors on the stack $\mfX^{(d,r)}$ that transfer between the relevant pairs of windows, and restrict to $T_F$ and $C_F^{-1}$ on $X_+^{(d,r)}$. To find these transfer functors, we embed the correspondences $Z^{(d,r,r+1)}$ into correpondences of Artin stacks
$$\begin{tikzcd}[column sep=12pt]   & \mfZ^{(d,r,r+1)} \drar{j}  \dlar[swap]{\pi} & \\  \mfX^{(d,r)} & & \mfX^{(d,r+1)}  \end{tikzcd} $$
in the same way that the correspondence \eqref{eqn.basiccorrespondence1} sits inside the correspondence \eqref{eqn.basiccorrespondencestack}.

When we discussed the 3-fold flop case in Section \ref{section.intro_sphericaltwists}, one of the ingredients that we needed in the proof was the locally-free resolution \eqref{eqn.basicfreeresolution1} of the sky-scraper sheaf $\cO_{V}$ on $\mfX$. We hit a similar step in our calculations in Section \ref{section.intro_sphericalcotwists}: we needed the locally-free resolutions \eqref{eqn.basicfreeresolution2} and \eqref{eqn.basicfreeresolution3} of two sheaves that lived on $\Im(j)\subset \Hom(V,V)$. In our proof of the general case we're going to need to generalize these examples, i.e. we're going to need to produce explicit locally-free resolutions of various sheaves that live on the unstable loci in $\mfX^{(d,r)}$. 

 These locally-free resolutions are very closely related to the exact sequences of bundles on Grassmannians that we mentioned at the end of Section \ref{section.intro_grassmannian}. For example, we already noted that when we restrict the resolution \eqref{eqn.basicfreeresolution1} to $X_-$ we get the pull-up of the Euler sequence on $\P V^\vee$. For another example, consider the exact sequence \eqref{eqn.G24complex1} on $X_-^{(4,2)}$. If we consider this as a complex on $\mfX^{(4,2)}$ then it is no longer exact, but we claim that it only fails to be exact at the last term, so it gives a resolution of a sheaf. The last two terms are the twist by $\cO(-1)$ of the map
$$\wedge^2 V (1) \To \cO $$
and the cokernel of this map is the sky-scraper sheaf along the unstable locus $\mathcal{U}=\mfX^{(4,2)} \backslash X_-^{(4,2)}$. So \eqref{eqn.G24complex1} arises from the locally-free resolution of $\cO_{\mathcal{U}}(-1)$. The other two sequences \eqref{eqn.G24complex2} and \eqref{eqn.G24complex3} on $\mfX^{(4,2)}$ arise from locally-free resolutions of more complicated sheaves supported on $\mathcal{U}$. 

These locally-free resolutions/exact sequences on Grassmannians do not appear to be very well-known. They are present implicitly in the book of Weyman \cite{Weyman:2007wu}, and most of the exact sequences were described explicitly in \cite{Fonarev:2011tq}. We describe them in excruciating detail in Appendix \ref{section.righthand_square}, as applications of Theorem \ref{theorem.freeresolutions}.

\section{Proofs}
\label{section.proofs}

\begin{notn}
For a Young diagram $\delta$ we write $\delta = (\cpt{\delta}{1}, \ldots, \cpt{\delta}{d})$ where the $\cpt{\delta}{\bullet}$ are the (non-increasing sequence of) row lengths of $\delta$. Trailing zeroes may be omitted. Given $V$ a vector space of dimension $d$, we write $\bbS^\delta V$ for the associated Schur power \cite{Weyman:2007wu}.

Let $\gamma$ be a Young diagram of width $\leq w$ and height $\leq h$, so that we can draw $\gamma$ inside a $w\times h$ rectangle. Take the complement of $\gamma$ inside this rectangle and rotate it by $180^{\circ}$s: this produces a new Young diagram which we denote by $\comp{w}{h}{\gamma}$. Figure \ref{figure.complement} gives an example.
\begin{figure}[h]
\begin{center}
\begin{tikzpicture}[auto,>=latex',scale=0.6]
\YoungDiag{Diagram}{0}{0}{4}{3}{{3,2}}{$\gamma=(3,2)$}
\YoungDiag{Complement}{7}{0}{4}{3}{{4,2,1}}{$\comp{w}{h}{\gamma}=(4,2,1)$}
\end{tikzpicture}
\end{center}
\caption{Young diagram $\gamma=(3,2)$ and its complement for $h=3$, $w=4$.}
\label{figure.complement}
\end{figure}
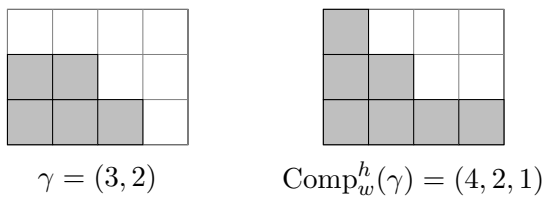
\end{notn}

\begin{rem}We will frequently use the following result from \cite[Exercise 2.18(a)]{Weyman:2007wu}:
$$\bbS^\gamma V^\vee \otimes \det V ^{\otimes w} = \bbS^{\comp{w}{d}{\gamma}} V.$$
\end{rem}

\subsection{Windows on Grassmannian flops}

\label{section.grassmannians}

Let $V$ be a vector space of dimension $d$, and $S$ be another vector space of dimension $r$, where $0<r\leq d$. For simplicity, we'll fix a trivialization of $\det V$ throughout.

Our first space is the affine Artin stack
$$\mfX = [\Hom(S,V)\oplus \Hom(V,S) \; /\; \operatorname{GL}(S) ].$$
In Section \ref{section.intro_grassmannian} we denoted this by $\mfX^{(d,r)}$, but from now on we will drop the $(d,r)$ from our notation. There are two possible GIT quotients of this stack, which correspond to open substacks denoted by
$$\begin{tikzcd}X_\pm \rar[hook]{i_{X_\pm}} & \mfX. \end{tikzcd}$$

\begin{rem} One quotient $X_+$ is the locus where the map from $S$ to $V$ is full rank: it is the total space of the vector bundle $\Hom(V,S)$ over $\Gr(r,V)$, where we reuse the notation $S$ to denote the tautological subspace bundle on the Grassmannian $\Gr(r,V)$.

Dually, $X_-$ is the locus where the map from $V$ to $S$ is of full rank: it is the total space of the vector bundle $\Hom(S,V)$ over the dual Grassmannian $\Gr(V, r)$, where now $S$ denotes the tautological quotient bundle.\end{rem}

 As anticipated in Section \ref{section.intro_window_shifts}, we are going to define some derived equivalences between $X_+$ and $X_-$ using \quotes{windows} in $D^b(\mfX$). To define these windows we need to recall Kapranov's exceptional collection for a Grassmannian \cite{Kapranov:2009wf}.

Let $\delta$ be a partition of some integer, which as usual we can draw as a Young diagram. Then associated to $\delta$ we have a Schur power $\bbS^\delta S^\vee$ of $S^\vee$. This is a representation of $\operatorname{GL}(S)$ and so induces a vector bundle on $\Gr(r,V)$. Now we put:

\begin{defn} $$\Gamma_{d,r} := \left\{ \text{Young diagrams $\gamma$ with height $\leq r$ and width $\leq d-r$ } \right\}$$ \end{defn}

Kapranov's exceptional collection for $\Gr(r,V)$ \cite{Kapranov:2009wf} is the set
$$ \setconds{ \bbS^\delta S^\vee }{ \delta\in \Gamma_{d,r} }.$$
We can also consider this as a set of vector bundles on $\Gr(V,r)$, on $X_{\pm}$, or on $\mfX$. These bundles give us our zeroth window, i.e. we define $\cW_0 \subset D^b(\mfX)$ as the full subcategory split-generated by this set of vector bundles. The other windows $\cW_k$ are obtained by tensoring $\cW_0$ by powers of the tautological line bundle
$$\cO(1) := \det S^\vee.$$

\begin{defn} $\cW_k$ is the full subcategory of $D^b(\mfX)$ split-generated by the set
$$ \setconds{ \bbS^\delta S^\vee(k) }{ \delta\in \Gamma_{d,r} }. $$
\end{defn}

Now observe that if the width of $\delta$ is strictly less than $d-r$ then we can create a new diagram $\tilde{\delta}\in \Gamma_{d,r}$ by adding on a new column of height $r$ to $\delta$ (see Figure \ref{fig.addcolumn}), and $$\bbS^\delta S^\vee (k) = \bbS^{\tilde{\delta}} S^\vee (k-1).$$

These are the bundles that lie in the generating set for both $\cW_k$ and the neighbouring window $\cW_{k-1}$. Since this observation will be useful in the sequel, we put:

\begin{defn}
$$\Gamma_{d,r}^{(1)} = \setconds{ \delta \in \Gamma_{d,r} }{ \operatorname{width}(\delta)<d-r } $$
\begin{equation*}\label{equation.splitGamma}\Gamma_{d,r}^{(2)} = \setconds{ \delta \in \Gamma_{d,r} }{ \operatorname{width}(\delta)=d-r } \end{equation*}
\end{defn}
We will also frequently switch between Schur powers of $S^\vee$ and of $S$. In terms of the latter, $\cW_k$ is generated by the set
$$ \setconds{ \bbS^\gamma S(d-r+k) }{ \gamma\in \Gamma_{d,r} }.$$

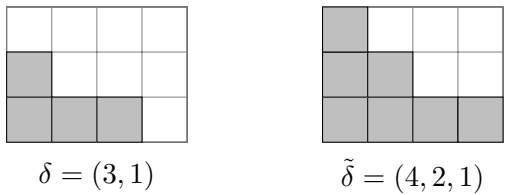
\begin{figure}[H]
\begin{center}
\begin{tikzpicture}[auto,>=latex',scale=0.6]
\YoungDiag{delta}{0}{0}{4}{3}{{3,1}}{$\delta=(3,1)$}
\YoungDiag{tilde_delta}{7}{0}{4}{3}{{4,2,1}}{$\tilde{\delta}=(4,2,1)$}
\end{tikzpicture}
\end{center}
\caption{Young diagram $\delta \in \Gamma_{d,r}^{(1)}$ and its twist $\tilde{\delta}$.}
\label{fig.addcolumn}
\end{figure}

The following proposition is the crucial ingredient in constructing our window equivalences:

\begin{prop}\label{prop.restrictionequivalences2} For any $k$ and $0<r<d$, both functors
$$i_{X_\pm}^*: \cW_k \To D^b(X_\pm)$$
are equivalences.
\end{prop}
\begin{proof}
By symmetry we only need the argument for $i_{X_+}^*$. We observe that $i_{X_+}^* S^\vee = \pi^* S^\vee$ where $\pi$ is the projection $X_+ \to \Gr(r,V)$. It immediately follows, because Schur powers commute with pullbacks, that $i_{X_+}^* \cW_k = \pi^* \mcT$, where $\mcT$ is the Kapranov tilting bundle for $\Gr(r,V)$ given by $$\mcT := \bigoplus_{\delta \in \Gamma_{d,r}} \bbS^\delta S^\vee.$$ An extended exercise in Schur functors \cite[Appendix C]{Donovan:2011ua} gives that $\pi^* \mcT$ is tilting on $X_+$.

It then suffices to show that the natural restriction map of derived functors $$\RHom_{\mfX}\left(\cW_k, \cW_k\right) \To \RHom_{X_+}\left(i_{X_+}^* \cW_k, i_{X_+}^* \cW_k\right) $$
induces isomorphisms on cohomology. There is no higher cohomology on the left-hand side because $\cW_k$ is locally free and the stack $\mfX$ is affine, and none on the right-hand side by the above tilting property. It therefore remains to show that the restriction map of \textit{ordinary} $\Hom$ functors $$\Hom_{\mfX}\left(\cW_k, \cW_k\right) \To \Hom_{X_+}\left(i_{X_+}^* \cW_k, i_{X_+}^* \cW_k\right)$$ is an isomorphism. Consider then the complement of $X_+$ in $\mfX$. This is the pull-up via the projection $\pi: \mfX \to \Hom(S,V)$ of the locus in $\Hom(S,V)$ consisting of maps of rank strictly less than $r=\dim S$. We then see from \cite[Prop 1.1(b)]{Bruns:198wwa} that its codimension is $d-r+1\geq2$, and hence the required isomorphism follows by normality.
\end{proof}

It immediately follows that:
\begin{thm}\label{thm.equivalences2} For $0<r<d$ there exists a {\em window equivalence} $\psi_k$ defined by the composition
$$\psi_k: \begin{tikzcd}[column sep=38pt] D^b(X_+) \rar{\sim}[swap]{(i^*_{X_+})^{-1}} & \cW_k \rar{\sim}[swap]{i^*_{X_-}} & D^b(X_-).\end{tikzcd}$$
\end{thm}
\begin{rem}Theorem \ref{thm.equivalences2} is obtained  in \cite[Section 5]{Buchweitz:2011ug} using a different method which works in arbitrary characteristic.\end{rem}

Consequently we have:
\begin{defn}\label{defn.window-autoequivalences2} We define {\em window-shift autoequivalences} $\omega_{k,l}$ by
$$\omega_{k,l} := \psi_k^{-1}\psi_l : D^b(X_+) \overset{\sim}{\To} D^b(X_+). $$
\end{defn}

\subsection{The geometric construction}
\label{section.geometry}
In \cite{Donovan:2011ua}, the first author constructed an endofunctor of $D^b(X_+)$ using more geometric techniques, and proved that it was an autoequivalence when $r\leq 2$. In this section we will show that  this endofunctor agrees with the window-shift $\omega_{0,1}$, and hence that it is in fact an autoequivalence for all $r$. 

In addition to the vector spaces $V$ and $S$, let $H$ be a third vector space, of dimension $r-1$. Consider the affine Artin stack
$$\mfY = [\Hom(H,V)\oplus \Hom(V, H) \; / \; \operatorname{GL}(H) ] $$
which is of course the same thing as $\mfX^{(d, r-1)}$. It contains an open substack $Y_+$ ($=X_+^{(d,r-1)}$) consisting of the locus where the map from $H$ to $V$ has full rank: this is the total space of a vector bundle over $\Gr(r-1,V)$.

Now let
$$\bar{\mfZ} = [\Hom(S, V) \oplus \Hom(V, H) \oplus \Hom(H, S) \; / \; \operatorname{GL}(H)\times \operatorname{GL}(S) ].$$
There are obvious maps
$$\mfY \overset{\pi}{\longleftarrow} \bar{\mfZ}  \overset{j}{\To}\mfX $$
given by composing the relevant two linear maps in $\bar{\mfZ}$, and then forgetting the redundant group action. This defines a correspondence between $\mfY$ and $\mfX$, however it is not exactly the correspondence that we want, rather we will define
$$\mfZ \subset \bar{\mfZ} $$
to be the substack where the map from $H$ to $S$ is an injection. This is the correspondence that we want to consider.

There is an open substack $Z \subset \mfZ$ where the map from $S$ to $V$ is also required to be an injection. We have a commutative diagram as follows:
\begin{equation}\label{diagram.correspondences}\begin{tikzcd} \mfY  & \mfZ \lar[swap]{\pi} \rar{j} & \mfX \\
Y_+ \uar[hookrightarrow]{i_{Y_+}} & Z \lar[swap]{\pi} \uar[hookrightarrow]{i_{Z_{}}} \rar{j} & X_+ \uar[hookrightarrow]{i_{X_+}} \end{tikzcd}\end{equation}

\begin{rem}The lower line of this diagram gives a correspondence between $Y_+$ and $X_+$. This was introduced in \cite{Donovan:2011ua} (in that paper $X_0$ denotes the space that we are calling $Y_+$, and $\hat{B}$ denotes the correspondence that we are calling $Z$), where it was used to construct endofunctors of $D^b(X_+)$ and $D^b(Y_+)$ as we shall we now explain. Note also that the correspondence is analogous to the Hecke correspondences in \cite{Cautis:2011wi}.\end{rem}

Consider the functor
 $$F := j_* \pi^* :  D^b(Y_+) \To D^b(X_+).$$
It has a right adjoint
$$R := \pi_* j^! :  D^b(X_+) \To D^b(Y_+)  $$
where 
$$j^!(-) = j^*(-)\otimes K_j[\dim j].$$
It also has a left adjoint
$$ L := \pi_*(j^*(-)\otimes K_\pi)[\dim \pi].$$
Both $X_+$ and $Y_+$ are Calabi-Yau \cite[Section 3.2]{Donovan:2011ua}, so $K_\pi=K_j$, so we deduce that
$$ R = L[-\sigma]$$
where 
$$\sigma = \dim \pi - \dim j = 2(d-r) +1.$$
This means that $F$ and $R$ are \textit{biadjoint} functors up to a shift. By applying standard Fourier-Mukai techniques \cite[Appendix A]{Donovan:2011ua} we may take cones on units and counits to give four endofunctors:
\begin{defn} \begin{enumerate} \item The \textbf{twist} functor $T_F : D^b(X_+) \To D^b(X_+)$ is the cone
$$ T_F := \cone{  FR \To \id  }.$$
It has a right adjoint
$$ T_F^\dagger := \cone{  \id \To FR[\sigma]  }.$$
\item The \textbf{cotwist} functor $C_F: D^b(Y_+) \To D^b(Y_+)$ is the cone
$$ C_F := \cone{  \id \To RF  }.$$
It has a right adjoint
$$C_F^\dagger := \cone{  RF[\sigma] \To \id  }[-1].$$
\end{enumerate} \end{defn}
General theory \cite{Anno:2010we} says that $T_F$ is an equivalence iff $C_F$ is an equivalence, given the fact that $X_+$ and $Y_+$ are Calabi-Yau. 

For $r\leq 2$ these functors were proven to be equivalences in \cite{Donovan:2011ua}: it is an immediate corollary of the following theorem that in fact $T_F$ is an equivalence for all $r<d$.

\begin{thm}\label{theorem.grassmantwist} For $r<d$, the twist functor $T_F$ is naturally isomorphic to the window-shift $\omega_{0,1}$. \end{thm}
\begin{proof}This proceeds formally from Lemmas \ref{lemma.twisttransferproperty1}, \ref{lemma.twisttransferproperty2} and \ref{lemma.twisttransferproperty3}, as in the proof of Proposition \ref{prop.shiftequalstwist1}.\end{proof}

Theorem \ref{theorem.grassmantwist} implies that $C_F$ is also an equivalence, and so $C_F^\dagger = C_F^{-1}$, at least for $r<d$. However we can say more: since $Y_+ =X_+^{(d,r-1)}$, we also have window-shift autoequivalences
$$\omega^Y_{k,l}:  D^b(Y_+) \To D^b(Y_+)  $$
and we prove:

\begin{thm}\label{theorem.grassmancotwist} The shifted inverse cotwist functor $C_F^{-1}[-\sigma]$ is naturally isomorphic to the window-shift $\omega^Y_{-1,0}$. \end{thm}
\begin{proof}This follows from Lemmas \ref{lemma.cotwisttransferproperty1}, \ref{lemma.cotwisttransferproperty2} and \ref{lemma.cotwisttransferproperty3}, once again using the method of proof in Proposition \ref{prop.shiftequalstwist1}.\end{proof}

Theorem \ref{theorem.grassmancotwist} was proved in \cite{Donovan:2011ua} for the $r\leq 2$ case.

We now temporarily reinstate the $d$'s and $r$'s into our notation, and state these theorems in a slightly different way. We have a whole chain of correspondences, going between $X_+^{(d,r)}$ and $X_+^{(d,r+1)}$ for every $r$ (see diagram \eqref{eqn.chainofcorrespondences}). So for every $d$ and $r$, we have both twist and cotwist endofunctors
$$T_F^{(d,r)}\:\: \mbox{and} \:\: C_F^{(d,r)}:  D^b \left(X_+^{(d,r)}\right) \overset{\sim}{\To} D^b\left(X_+^{(d,r)}\right).$$
We also have our window-shift autoequivalences $\omega_{k,l}^{(d,r)}$ of $D^b\left(X_+^{(d,r)}\right)$, and our result is that
$$T_F^{(d,r)}= \omega_{0, 1}^{(d,r)}, $$
$$C_F^{(d,r)}[-\sigma_{d,r}] = \omega_{0,-1}^{(d,r)}, $$
where $\sigma_{d,r}=2(d-r)-1$. This means these two functors are almost inverse to each other. More precisely, the relations \eqref{equation.windowshiftrelations} between window-shifts imply:

\begin{cor}\label{corollary.twistcotwist}
$$\left(T_F^{(d,r)}\right)^{-1} = (\otimes \cO(1))\circ C_F^{(d,r)}[-\sigma_{d,r}] \circ (\otimes \cO(-1) ) $$
\end{cor}

\begin{notn} Finally we introduce a little more notation. On $\mfZ$ there are two tautological line bundles given by the determinants of $S$ and $H$. To distinguish between them we put
\begin{align*} \cO(1) & := (\det S)^\vee, \\
\cO\br{1} & := (\det H)^\vee. \end{align*}
\end{notn}

\subsubsection{Analysis of correspondences}
\label{section.analysingcorrespondences}

Before we begin the proofs of our two theorems, let us make an observation about the diagram \eqref{diagram.correspondences}. The left-hand square is trivial in the $\Hom(V,H)$ directions, and at various points in the following proofs we will be considering sheaves and maps that are constant over these trivial directions. Therefore it is helpful to introduce the notation
\begin{equation}\label{diagram.PQsquare}\begin{tikzcd} \mfQ &  \mfP \lar[swap]{\pi} \\ Q \uar[hookrightarrow]{i_Q} & P \uar[hookrightarrow, swap]{i_P} \lar[swap]{\pi} \end{tikzcd}\end{equation}
for the square that we obtain by deleting the $\Hom(V,H)$ directions, so for example $\mfQ$ is the stack $[\Hom(H,V) \;/\; \operatorname{GL}(H)]$. 

Now look at the right-hand square in \eqref{diagram.correspondences}. This square is a fibre product, with the map
$j: Z\to X_+$ being just the restriction of $j:\mfZ \to \mfX$ to the open substack $X_+\subset \mfX$. Furthermore, the map $j:\mfZ\to \mfX$ is trivial in the $\Hom(S,V)$ directions, and removing them gives a map which we denote
$$ j: \mfS \To \mfT.$$
We do some more analysis of these two squares in the appendix.

\subsubsection{Twist}

We now turn to the proof of Theorem \ref{theorem.grassmantwist}. The structure of the proof is as outlined in Section \ref{section.intro_sphericaltwists}: the result follows formally from the existence of a transfer functor, as in the proof of Proposition \ref{prop.shiftequalstwist1}. The functor we need to consider is
$$T_\mfF := \cone{ \mfF\mfR \To \id  } : D^b(\mfX) \To D^b(\mfX)  $$
where
$$\mfF := j_* \pi^* : D^b(\mfY) \To D^b(\mfX) $$
and 
$$\mfR :=\pi_* j^! :  D^b(\mfX) \To D^b(\mfY).$$
Then we just need to establish the following three properties:

\begin{lem} \label{lemma.twisttransferproperty1} $i_{X_-}^* T_\mfF = i_{X_-}^*$
\end{lem}
\begin{lem}\label{lemma.twisttransferproperty2} $T_\mfF$ maps the window $\cW_1$ to the window $\cW_0$.
\end{lem}
\begin{lem} \label{lemma.twisttransferproperty3} The following diagram commutes:
 $$\begin{tikzcd}
\cW_1 \rar{T_\mfF} \dar[swap]{i_{X_+}^*} & D^b(\mfX) \dar{i_{X_+}^*} \\ 
 D^b(X) \rar{T_F} & D^b(X) 
\end{tikzcd}$$\end{lem}

Lemma \ref{lemma.twisttransferproperty1} is obvious, since the image of $j$ is exactly the unstable locus that gets deleted to form $X_-$. The remaining two lemmas are rather more involved.

\begin{proof}[Proof of Lemma \ref{lemma.twisttransferproperty2}]
We need to calculate the effect of $T_\mfF=[\mfF\mfR \to \id]$ on the vector bundles
$$\setconds{ \bbS^\delta S^\vee(1) }{ \delta\in \Gamma_{d,r} }$$
and verify that each one ends up in the window $\cW_0$.  Recall that
$$\mfF\mfR = j_*\pi^*\pi_* j^!.$$
We will send our vector bundles through each of these functors in turn. The first one is
  \begin{equation}\label{equation.j^!}j^!(-) = (K_j\otimes j^*(-))[\dim j] = j^*(-)(r-d-1)\br{d-r}[r-d-1]. \end{equation}
Note that the calculation of the canonical bundle here is straightforward as the spaces involved are open substacks of quotients of vector spaces: recall also that $\det V$ is trivialized. Let $\gamma = \comp{d-r}{r}{\delta}$, so
$$\bbS^\delta S^\vee (1)= \bbS^\gamma S(d-r+1)$$ 
and then we have
\begin{equation}\label{equation.j^!onbundle} j^! (\bbS^\delta S^\vee(1)) = \bbS^\gamma S \br{d-r} [r-d-1].\end{equation}
Next we apply $\pi_*$ to this object, and by the projection formula it suffices to know what $\pi_* \bbS^\gamma S$ is. Everything here is constant along the $\Hom(V,H)$ directions in $\mfZ$ and $\mfY$, so Proposition \ref{prop.pushdownsbypi}(i) tells us that 
$$\pi_* \bbS^\gamma S = \bbS^\gamma H$$
 and so
\begin{equation}\label{equation.pi_*Rj^!}
 \pi_* j^! (\bbS^\delta S^\vee(1)) = \bbS^\gamma H \br{d-r} [r-d-1]. 
\end{equation}
This expression will be zero iff the height of $\gamma$ is $r$, or equivalently iff the width of $\delta$ is less than $d-r$. So if $\delta\in \Gamma_{d,r}^{(1)}$ then 
$$T_\mfF(\bbS^\delta S^\vee(1)) = \bbS^\delta S^\vee(1).$$
Since these bundles are the ones that also lie in the target window $\cW_0$, we have verified the lemma on this subset of our generating set. 

For the remaining bundles, we continue with the calculation of $T_\mfF$. Let $\delta\in \Gamma_{d,r}^{(2)}$, so the height of $\gamma$ is $<r$ and we can define $\hat{\delta} = \comp{d-r}{r-1}{\gamma}$, which is $\delta$ with its first row deleted. Then
$$\bbS^\gamma H \br{d-r} = \bbS^{\hat{\delta}} H^\vee$$
and 
$$\mfF\mfR (\bbS^\delta S^\vee(1) ) = j_*(\bbS^{\hat{\delta}} H^\vee) [r-d-1].$$
This is a (shift of a) torsion sheaf on $\mfX$. Since everything here is constant in the $\Hom(S,V)$ directions, the sheaf is evidently the pull-up of the sheaf $j_*(\bbS^{\hat{\delta}} H^\vee)$ from the stack $\mfT$. In Appendix \ref{appendix} (Theorem \ref{theorem.freeresolutions}) we construct a free resolution of $j_*(\bbS^{\hat{\delta}} H^\vee)$, of length $d-r+1$. If we pull this resolution up to $\mfX$ and shift it by $[r-d-1]$ then we get a complex of vector bundles situated in non-negative degrees
\begin{equation}\label{equation.complexforFR} \complex{ \begin{tikzcd}[baseline=-5pt] \underline{\bbS^{\hat{\delta}_{K}} S^\vee \otimes \Wedge^{s_{K}} V}  \rar \nextcell \ldots  \rar \nextcell\bbS^{\hat{\delta}_1} S^\vee \otimes \Wedge^{s_1} V  \rar \nextcell   \bbS^{\hat{\delta}_0} S^\vee \end{tikzcd} } \end{equation}
which is quasi-isomorphic to $\mfF\mfR (\bbS^\delta S^\vee(1))$. Remarks \ref{remark.combinatorics} following Theorem \ref{theorem.freeresolutions} tell us more about the terms in this complex. Firstly, by Remark \ref{remark.combinatorics}(v), $\hat{\delta}_K$ is the Young diagram of height $r$ and width $d-r+1$ such that if we delete the first row and the first column we get back $\hat{\delta}$, so we find that
$$\bbS^{\hat{\delta}_K} S^\vee = \bbS^\delta S^\vee (1).$$
Also $s_K=d$ by Remark \ref{remark.combinatorics}(iii), and $\det(V)$ is trivialized, so the term in degree zero is $\bbS^\delta S^\vee (1)$. Secondly, for $k<K$ we have $\hat{\delta}_k\in \Gamma_{d,r}$ by Remark \ref{remark.combinatorics}(iv), so every term in positive degree lies in the target window $\cW_0$.

From this description, we see that the natural map $\mfF\mfR (\bbS^\delta S^\vee(1)) \to \bbS^\delta S^\vee(1)$ is given by some non-zero map of bundles
$$\iota: \bbS^\delta S^\vee(1) \To \bbS^\delta S^\vee(1) $$
since there are no higher Ext groups between vector bundles on the stack $\mfX$. This map arises, via adjunction, from the natural map
$$\pi^*\pi_*j^!(\bbS^\delta S^\vee(1)) \To j^!(\bbS^\delta S^\vee(1)) $$
which is (by \eqref{equation.j^!onbundle} and \eqref{equation.pi_*Rj^!}) a shift and twist of a map
$$\bbS^\gamma H \To \bbS^\gamma S.$$
It follows from the proof of Proposition \ref{prop.pushdownsbypi} that this is actually the tautological map, so in particular it is constant over the $\Hom(S,V)$ directions. Therefore the map $\iota$ must also be constant in those directions, and so is the pull-up of a map that lives on the smaller stack $\mfS$. By Lemma \ref{lemma.1dmappingspaces}, it must be an isomorphism. Consequently, the cone $T_\mfF(\bbS^\delta S^\vee (1))$ is quasi-isomorphic to the positive-degree part of the complex \eqref{equation.complexforFR}, so it lives in the target window $\cW_0$.
\end{proof}

\begin{proof}[Proof of Lemma \ref{lemma.twisttransferproperty3}]
Recall the diagram \eqref{diagram.correspondences}. $T_\mfF$ is the cone $[j_*\pi^*\pi_*j^! \to \id ]$ of endofunctors of $D^b(\mfX)$, and $T_F$ is the same cone of endofunctors of $D^b(X_+)$. We wish to compare $i_{X_+}^* T_\mfF$ with $T_F i_{X_+}^*$. 

The right-hand square in \eqref{diagram.correspondences} is a fibre square and the open inclusion $i_{X^+}$ is flat, so
$$ i_{X_+}^* j_* = j_* i_Z^* $$
and hence 
$$ i_{X_+}^*j_*\pi^*\pi_*j^! = j_*i_Z^*\pi^*\pi_*j^!  = j_*\pi^*i_{Y_+}^*\pi_*j^!.$$
The left-hand square is not a fibre square, so we have only a natural transformation
$$ \tau: i_{Y_+}^*\pi_* \To \pi_* i_Z^*.$$
 Then for any $\mathcal{E}\in D^b(\mfX)$ we get a morphism
\begin{equation} \label{equation.nattranstau}
\begin{tikzcd}[column sep=70pt] j_*\pi^*i_{Y_+}^*\pi_*j^!\cE \rar{j_*\pi^*(\tau_{j^!\cE})} & j_*\pi^*\pi_*i_Z^*j^!\cE = j_*\pi^*\pi_*j^!i_{X_+}^*\cE \end{tikzcd}
\end{equation}
where the final equality holds because $i_{X_+}$ and $i_Z$ are open inclusions. Thus we have a square
$$\begin{tikzcd} i_{X_+}^*j_*\pi^*\pi_*j^!\cE   \rar  \dar[swap]{j_*\pi^*(\tau_{j^!\cE})} &  i_{X_+}^*\cE \dar[equals] \\
j_*\pi^*\pi_*j^!i_{X_+}^*\cE \rar &  i_{X_+}^*\cE \end{tikzcd}$$
which commutes by naturality of adjunctions. This means we have a natural transformation from $i_{X_+}^*T_\mfF$ to $T_F i_{X_+}^*$. We claim that this becomes a natural isomorphism when we restrict it to the window $\cW_1$. It is sufficient to check this on the generating vector bundles, i.e. we just need to check that \eqref{equation.nattranstau} is an isomorphism when $\cE$ is a vector bundle $\bbS^\gamma S (d-r+1)$ for some $\gamma \in \Gamma_{d,r} $. By \eqref{equation.j^!} in the proof of Lemma \ref{lemma.twisttransferproperty2} we know that $j^!\cE$ is a shift of the bundle 
$$\bbS^\gamma S\br{d-r}$$
so it is sufficient to prove that
$$i_{Y_+}^*\pi_* \bbS^\gamma S \xrightarrow{\tau_{\bbS^\gamma S}} \pi_* i_Z^* \bbS^\gamma S $$
is an isomorphism. Everything here is constant in the $\Hom(V,H)$ directions, so we can actually work on the smaller square \eqref{diagram.PQsquare}, and Proposition \ref{prop.pushdownsbypi}(ii) is the required statement.\end{proof}

\subsubsection{Cotwist}

Now we prove Theorem \ref{theorem.grassmancotwist}. The structure of the proof is exactly the same, although curiously enough in this case the transfer functor is
$$\mfR\mfF : D^b(\mfY) \To D^b(\mfY),$$
i.e. we do not take the cone from the identity. 

We denote the windows on $\mfY$ by $\cV_k$, where each $\cV_k$ is split-generated by the set
$$\setconds{ \bbS^\delta H^\vee \br{k} }{ \delta\in \Gamma_{d, r-1} }.$$
The three lemmas that we need are:

\begin{lem} \label{lemma.cotwisttransferproperty1} $i_{Y_-}^* \mfR\mfF = i_{Y_-}^*$ when restricted to the window $\cV_0$.
\end{lem}
\begin{lem}\label{lemma.cotwisttransferproperty2} $\mfR\mfF$ maps the window $\cV_0$ to the window $\cV_{-1}$.
\end{lem}
\begin{lem} \label{lemma.cotwisttransferproperty3} The following diagram commutes:
 $$\begin{tikzcd}[column sep=35pt, row sep=35pt]
\cV_0 \rar{\mfR\mfF} \dar[swap]{i_{Y_+}^*} & D^b(\mfY) \dar{i_{Y_+}^*} \\ 
 D^b(Y_+) \rar{C_F^{\dagger}[-\sigma]} & D^b(Y_+) 
\end{tikzcd}$$\end{lem}

Notice that unlike the twist case the first lemma is not obvious, and only holds on the window $\cV_0$ and not on the whole of $D^b(\mfY)$. We will prove the first two lemmas in reverse order.

\begin{proof}[Proof of Lemma \ref{lemma.cotwisttransferproperty2}]
We compute the effect of $\mfR\mfF$ on the generating vector bundles
$$\setconds{ \bbS^{\delta} H^\vee }{ \delta\in \Gamma_{d,r-1} }$$
of the window $\cV_0$. By Corollary \ref{corollary.j^!j_*} we have that $j^!j_*\pi^* \bbS^{\delta} H^\vee$ is the complex
\begin{equation}\label{equation.j^!j_*pi^*ofH} \complex{ \begin{tikzcd}[column sep=14pt]  \underline{\bbS^{\epsilon_{K}} S\br{d-r}\otimes \Wedge^{s_{K}} V}   \rar \nextcell \ldots \rar \nextcell  \bbS^{\epsilon_1} S \br{d-r} \otimes \Wedge^{s_1} V \rar \nextcell   \bbS^{\epsilon_0} S\br{d-r} \end{tikzcd} } \end{equation}
where $K=d-r+1$, and the partitions $\epsilon_k$ and the numbers $s_k$ are defined by Algorithm \ref{algorithm} and \eqref{equation.epsilon_defn}.

 Suppose that $\delta\in \Gamma_{d,r-1}^{(1)}$, i.e. width$(\delta)<d-r+1$. Then Remark \ref{remark.combinatoricsdual}(iv*) tells us that all the $\epsilon_k$'s except for $\epsilon_K$ have height $r$, and Remark \ref{remark.combinatoricsdual}(iii*) that $s_K=d$ and hence $\Wedge^{s_{K}} V$ is trivial.  So by Proposition \ref{prop.pushdownsbypi} applying $\pi_*$ to \eqref{equation.j^!j_*pi^*ofH} kills all the terms in positive degree, and leaves only
$$ \pi_*(\bbS^{\epsilon_K} S)\br{d-r} = \bbS^{\epsilon_K} H \br{d-r}  = \bbS^\delta H^\vee $$
where the final equality is because of Remark \ref{remark.combinatoricsdual}(v*). So if $\delta\in \Gamma_{d,r-1}^{(1)}$ then
$$\mfR\mfF(\bbS^\delta H^\vee) = \bbS^\delta H^\vee.$$ 
These bundles already lie in the target window $\cV_{-1}$, so this verifies the lemma on this subset.

Now take $\delta \in  \Gamma_{d,r-1}^{(2)}$. By Proposition \ref{prop.pushdownsbypi} again we have that for any $k$
$$\pi_*(\bbS^{\epsilon_k} S)\br{d-r} = \bbS^{\epsilon_k} H \br{d-r} = \bbS^{\hat{\delta}_k} H^\vee\br{-1}$$
where $\hat{\delta}_k = \comp{d-r+1}{r-1}{\epsilon_k}$, which is well-defined by Remark \ref{remark.combinatoricsdual}(vi*). So we have represented $\mfR\mfF(\bbS^{\delta} H^\vee)$ by a complex of bundles, with each term lying in the target window $\cV_{-1}$.
\end{proof}

\begin{proof}[Proof of Lemma \ref{lemma.cotwisttransferproperty1}]

Composing $i_{Y_-}^*$ with the unit of the adjunction gives a natural transformation
\begin{equation}\label{eqn.unitofadjunction}  i_{Y_-}^* \To i_{Y_-}^*\mfR\mfF.\end{equation}
It is sufficient to show that the components of this natural transformation are isomorphisms on the generating set of vector bundles for $\cV_0$. 

Pick  $\delta \in \Gamma_{d,r-1}$. We know that $j^!j_*\pi^*\bbS^{\delta}H^\vee$ is given by the complex \eqref{equation.j^!j_*pi^*ofH}, so the unit of the $j_*$-$j^!$ adjunction is given by some map of bundles on $\mfZ$
$$\eta:  \bbS^{\delta} H^\vee \To \bbS^{\epsilon_K}S\br{d-r}\otimes \Wedge^{s_K} V.$$ Furthermore this map is constant over the $\Hom(S,V)$ directions, i.e. it is pulled-up from the stack $\mfS$.

If $\delta$ lies in $\Gamma_{d,r-1}^{(1)}$ then we have $s_K=d$ and $\bbS^{\delta}H^\vee = \bbS^{\epsilon_K} H\br{d-r}$, so by Lemma \ref{lemma.1dmappingspaces}(ii) the map $\eta$ must be a twist of the tautological map, up to a scalar. Therefore the adjoint map under the $\pi^*$-$\pi_*$ adjunction given by
$$\bbS^\delta H^\vee \To \mfR\mfF(\bbS^\delta H^\vee ) = \bbS^\delta H^\vee $$
is a scalar multiple of the identity. So on this subset of the generating bundles we have shown that \eqref{eqn.unitofadjunction} is an isomorphism even before restricting to $Y_-$.

 Now let $\delta\in \Gamma_{d,r-1}^{(2)}$. As we argued in the proof of Lemma \ref{lemma.cotwisttransferproperty2}, applying $\pi_*$ to the complex \eqref{equation.j^!j_*pi^*ofH} shows that  $\mfR\mfF(\bbS^\delta H^\vee)$ is given by a complex
\begin{equation*} \complex{ \begin{tikzcd}[column sep=16pt] \bbS^{\hat{\delta}_{K}} H^\vee\otimes \Wedge^{s_{K}} V \br{-1}   \rar \nextcell \ldots\rar \nextcell  \bbS^{\hat{\delta}_1} S \otimes \Wedge^{s_1} V \br{-1} \rar\nextcell   \underline{\bbS^{\hat{\delta}_0} H^\vee\br{-1}} \end{tikzcd} }. \end{equation*}
The diagrams $\hat{\delta}_k$ arose in the following way. Starting from $\delta_0=\delta$, we applied Algorithm \ref{algorithm} to get a sequence of diagrams $\delta_k$. These all have width $d-r+1$, and so using \eqref{equation.epsilon_defn} we have
$$\hat{\delta}_k =  \comp{d-r+1}{r-1}{ \comp{d-r+1}{r}{\delta_k} }.$$
Hence $\hat{\delta}_k$ is the diagram obtained from $\delta_k$ by deleting the first row. This means that if we apply Algorithm \ref{algorithm} to the starting diagram $\hat{\delta}_0$, and with the parameter $r$ replaced by $r-1$, then it produces the sequence of diagrams $\hat{\delta}_k$. 

Now recall that $\mfY$ is the analogue of $\mfX$ but with $r$ replaced by $r-1$. Therefore by Theorem \ref{theorem.freeresolutions} there is a complex of bundles on $\mfY$
\begin{equation*} \complex{ \begin{tikzcd}[column sep=22pt] \bbS^{\hat{\delta}_{K+1}} H^\vee\otimes \Wedge^{s_{K+1}} V \rar{\tilde{\eta}} \nextcell  \bbS^{\hat{\delta}_{K}} H^\vee\otimes \Wedge^{s_{K}} V   \rar \nextcell \ldots \rar \nextcell \bbS^{\hat{\delta}_1} H^\vee \otimes \Wedge^{s_1} V  \rar\nextcell  \underline{\bbS^{\hat{\delta}_0} H^\vee} \end{tikzcd}} \end{equation*}
which is a free resolution of a sheaf supported on the unstable locus that we remove when we form $Y_{-}$. Furthermore, since width$(\hat{\delta}_0)< d-r+2$, Remark \ref{remark.combinatorics}(iii) tells us that $s_{K+1} = d$, and
$$ \bbS^{\hat{\delta}_{K+1}} H^\vee \br{-1} = \bbS^{\delta} H^\vee $$ 
because removing the first column of $\hat{\delta}_{K+1}$ gives $\delta$ by Remark \ref{remark.combinatorics}(v) and $\hat{\delta}_{K+1}$ has height $r-1$ by Remark \ref{remark.combinatorics}(ii). So we have found a map on $\mfY$
$$\tilde{\eta}:  \bbS^\delta H^\vee \To \bbS^{\hat{\delta}_{K}} H^\vee\otimes \Wedge^{s_{K}} V \br{-1} $$
which induces a quasi-isomorphism
$$ i_{Y_-}^*  \bbS^\delta H^\vee \To i_{Y_-}^*\mfF\mfR( \bbS^\delta H^\vee).$$
We claim that $\tilde{\eta}$ is the adjoint to $\eta$ under the $\pi^*$-$\pi_*$ adjunction, at least up to a scalar factor. If we can show this claim then the proof of the lemma is complete, because then applying \eqref{eqn.unitofadjunction} to $\bbS^\delta H^\vee$ gives the above quasi-isomorphism.

 To show the claim, observe that the adjoint of $\tilde{\eta}$ is given by the composition
$${\renewcommand{\arraystretch}{1.4} \begin{array}{rcl}
\bbS^\delta H^\vee & \stackrel{\tilde{\eta}}{\To}  & \bbS^{\hat{\delta}_{K}} H^\vee\otimes \Wedge^{s_{K}} V \br{-1} \\
& =\!= & \bbS^{\epsilon_K} H\br{d-r}\otimes \Wedge^{s_{K}} V \\
& \stackrel{\tau}{\To} & \bbS^{\epsilon_K} S\br{d-r}\otimes \Wedge^{s_{K}} V\end{array}}$$
on the stack $\mfZ$, where $\tau$ is the tautological map. By construction, $\tau\tilde{\eta}$ is independent of the $\Hom(S,V)$ directions, so it is pulled-up from $\mfS$. Also, both $\eta$ and $\tau\tilde{\eta}$ must be $\operatorname{SL}(V)$-equivariant, because our entire construction is, so by Lemma \ref{lemma.1dmappingspaces3} they agree up to a scalar factor.
\end{proof}

\begin{proof}[Proof of Lemma \ref{lemma.cotwisttransferproperty3}]
The functor $C_F^{\dagger}[-\sigma]$ is a (shifted) cone on the natural transformation
$$ RF \To \id[-\sigma].$$
The arguments in the proof of Lemma \ref{lemma.twisttransferproperty3} show that there is a natural transformation
$$i_{Y_+}^* \mfR\mfF \To RF i_{Y_+}^*.$$
We claim this induces a natural isomorphism
$$ i_{Y_+}^*\mfR\mfF \To C_F^{\dagger}[-\sigma] i_{Y_+}^*$$
of functors from $\cV_0$ to $D^b(Y_+)$, i.e. for every object $\mathcal{E}\in \cV_0$ the two natural morphisms 
\begin{equation}\label{equation.exacttriangle} i_{Y_+}^* \mfR\mfF(\mathcal{E}) \To RF i_{Y_+}^* (\mathcal{E}) \To i_{Y_+}^*\mathcal{E}[-\sigma] \end{equation}
form (two-thirds of) an exact triangle. It is sufficient to prove this claim on the generating set of vector bundles.

Fix  $\bbS^\delta H^\vee \in \cV_0$. Arguing again as in  Lemma \ref{lemma.twisttransferproperty3}, there is a natural isomorphism
$$ i_{Z}^* j^!j_*\pi^* \bbS^\delta H^\vee  \iso  j^!j_*\pi^*  i_{Y_+}^* \bbS^\delta H^\vee.$$
Combining this with a component of the natural transformation from $i_{Y_+}^*\pi_* $ to $\pi_*i_Z^*$ gives us the natural morphism
\begin{equation}\label{eqn.curlyRFtoRF}i_{Y_+}^*\mfR\mfF(\bbS^\delta H^\vee) \To RFi_{Y_+}^*(\bbS^\delta H^\vee).\end{equation}
By Corollary \ref{corollary.j^!j_*}, $j^!j_*\pi^*\bbS^\delta H^\vee$ is a complex
\begin{equation*} \complex{ \begin{tikzcd}[column sep=18pt] \underline{\bbS^{\epsilon_{K}} S \br{d-r} \otimes \Wedge^{s_{K}} V} \rar \nextcell\ldots \rar \nextcell  \bbS^{\epsilon_1} S \br{d-r} \otimes \Wedge^{s_1} V \rar \nextcell   \bbS^{\epsilon_0} S\br{d-r} \end{tikzcd} }.  \end{equation*} We can understand the morphism \eqref{eqn.curlyRFtoRF} term-by-term in this complex, i.e. it is the aggregate of the natural maps
$$i_{Y_+}^*\pi_* \bbS^{\epsilon_k} S \To\pi_*i_Z^*  \bbS^{\epsilon_k} S $$
twisted by powers of $\cO\br{1}$ and exterior powers of $V$. These maps are constant in the $\Hom(V,H)$ directions in $Y_+$.

Now consider the natural transformation from $RF$ to $\id[-\sigma]$. It arises in the following way. For any object $\cE\in D^b(Y_+)$, the natural morphism
$$j^*j_*\pi^* \cE \To \pi^*\cE$$
induces a morphism
$$j^!j_*\pi^*\cE \To \pi^!\cE  [-\sigma]$$
because $\sigma = \dim  \pi - \dim j$ by definition and the relative canonical bundles $K_\pi$ and $K_j$ are equal, as $\mfX$ and $\mfY$ are Calabi-Yau. Then the $\pi_*$-$\pi^!$ adjunction gives the morphism
$$RF(\cE) = \pi_* j^! j_* \pi^* \cE \To \cE[-\sigma].$$
We apply this to the case $\cE = i_{Y_+}^* \bbS^\delta H^\vee$. We know that $j^*j_*\pi^*\bbS^\delta H^\vee$ is a complex
\begin{equation*}  \complex{ \begin{tikzcd}  \bbS^{\delta_{K}} S^\vee \otimes \Wedge^{s_{K}} V  \rar \nextcell\;\;   \ldots\;\;  \rar \nextcell\bbS^{\delta_1} S^\vee \otimes \Wedge^{s_1} V  \rar \nextcell  \underline{\bbS^{\delta} S^\vee} \end{tikzcd} }. \end{equation*}
Hence the natural map from $j^*j_*\pi^*\bbS^\delta H^\vee$ to $\pi^*\bbS^\delta H^\vee$ must be given by some non-zero map of bundles
$$ \bbS^\delta S^\vee \To\pi^* \bbS^\delta H^\vee.$$
This map must be independent of the $\Hom(S,V)$ directions in $Z$ because both $j$ and $\pi^*\bbS^\delta H^\vee$ are, i.e. it is the pull-up of a map from the stack $\mfS$. So by Lemma  \ref{lemma.1dmappingspaces}(ii) it must be the tautological map (up to a scalar multiple), and in particular it is also constant over the $\Hom(V,H)$ directions in $Z$. Consequently, the natural map 
$$j^!j_*\pi^*\bbS^\delta H^\vee \To \pi^!\bbS^\delta H^\vee [-\sigma] $$
is given by a map of bundles
$$\bbS^{\epsilon_0} S\br{d-r} \To \pi^!\bbS^\delta H^\vee [-\sigma] $$
where $\epsilon_0 = \comp{d-r+1}{r}{\delta}$ as before, and the natural map from $RFi_{Y_+}^*(\bbS^\delta H^\vee)$ to $i_{Y_+}^*\bbS^\delta H^\vee[-\sigma]$ is obtained by restricting this map to $Z$ and taking its adjoint. Note that
$$\bbS^\delta H^\vee = \bbS^{\tilde{\epsilon}_0} H \br{d-r+1} $$
where $\tilde{\epsilon}_0 = \comp{d-r+1}{r-1}{\delta}$ and so $\tilde{\epsilon}_0$ is $\epsilon_0$ with its first row removed.

Now we evaluate \eqref{equation.exacttriangle} on the object $\bbS^\delta H^\vee$ and verify that we obtain an exact triangle as required. Combining the above arguments, the result can be written as the twist by $\cO\br{d-r}$ of a diagram as follows:
$$\begin{tikzcd}    i_{Y_+}^*\pi_* \bbS^{\epsilon_K}S \otimes \Wedge^{s_{K}} V  \rar \dar&\ldots\rar &  i_{Y_+}^*\pi_* \bbS^{\epsilon_1}S \otimes \Wedge^{s_1} V  \rar \dar&  i_{Y_+}^*\pi_* \bbS^{\epsilon_0}S  \dar \\
    \pi_* i_Z^*\bbS^{\epsilon_K}S \otimes \Wedge^{s_{K}} V  \rar &\ldots\rar &  \pi_* i_Z^*\bbS^{\epsilon_1}S \otimes \Wedge^{s_1} V  \rar&  \pi_* i_Z^* \bbS^{\epsilon_0} S \dar \\
            &                                                                                                                                &     &     i_{Y_+}^*\bbS^{\tilde{\epsilon}_0} H\br{1}[-\sigma] 
\end{tikzcd}$$
All the vertical arrows are constant in the $\Hom(V,H)$ directions, so we can analyse them on the smaller space $Q$.  By Remark \ref{remark.combinatoricsdual}(ii*), the width of $\epsilon_0$ is $d-r+1$, and the width of $\epsilon_k$ is at most $d-r$ for $k>0$. Then by Proposition \ref{prop.pushdownsbypi}(ii) and (iii) (and the discussion following), the first $K$ columns of the above diagram are isomorphisms, and the final column gives an exact triangle. So  \eqref{equation.exacttriangle} yields an exact triangle on each object $\bbS^\delta H^\vee\in \cV_0$ as claimed.
\end{proof}

\appendix
\section{}
\label{appendix}

In this appendix we study the behaviour of tautological vector bundles as we push them around the two squares in the diagram \eqref{diagram.correspondences} as discussed in Section \ref{section.analysingcorrespondences}.

\begin{notn} As in the rest of the paper, we let  $V$, $S$ and $H$ be vector spaces of dimensions $d$, $r$ and $r-1$ respectively, under the assumption that $d \geq r > 0$. We also fix a trivialization of $\det V$.
\end{notn}

\subsection{The left-hand square}

Let $\mfQ$ be the affine stack
$$\mfQ = [\Hom(H,V) \; / \; \operatorname{GL}(H)]$$
and define a second stack
$$ \mfP \subset [\Hom(H, S)\oplus \Hom(S,V) \; / \; \operatorname{GL}(H)\times \operatorname{GL}(S)] $$
to be the locus where the map from $H$ to $S$ is an injection. Then we have a composition map $\pi: \mfP \to \mfQ$. 

Let $P\subset \mfP $ be the locus where both maps are injections, so $P$ is a partial flag variety. The image of 
$P$ under $\pi$ is the Grassmannian $Q \cong \Gr(r-1, V)$ which forms an open substack of $\mfQ$. So we have a commutative diagram as follows: 
$$\begin{tikzcd} \mfQ & \mfP \lar[swap]{\pi}\\ Q \uar[hook]{i_Q} & P \uar[hook]{i_P} \lar[swap]{\pi} \end{tikzcd}$$
If we choose a point $q\in Q$ then the fibre of $P$ over $q$ is the projective space $\P(V/H)$. The fibre of $\mfP$ over $q$ is slightly larger: it is given by the affine stack
\begin{equation}\label{equation.fibreofP} [\Hom(L, V/H ) \; / \; \operatorname{GL}(L)] \end{equation}
where $L$ is the one-dimensional space $S/H$. In particular the above diagram is not a fibre square. The fibres of $\mfP$ over points  $q\in \mfQ \backslash Q$ have the same description, although the dimension of $V/H$ will jump.

Observe also that the relative canonical bundle of $\pi$ is
$$ K_\pi = (\det S)^{d-r+1}\otimes (\det H)^{r-d} = L^{\otimes d-r+1}\otimes \det H$$
and the dimension of $\pi$ is $d-r$.

\begin{prop}\label{prop.pushdownsbypi}
Let $\gamma$ be a Young diagram, and let $\bbS^\gamma S$ be the associated vector bundle on $\mfP$. 
\begin{enumerate}
 \item We have an isomorphism of bundles on $\mfQ$
$$\bbS^\gamma H \iso \pi_* \bbS^\gamma S.$$
\item
 If the width of $\gamma$ is at most $d-r$ then we have an isomorphism of bundles on $Q$
$$\bbS^\gamma H =  i_Q^* \pi_* \bbS^\gamma S \iso \pi_* i_P^* \bbS^\gamma S,$$ i.e. we have base change for $\bbS^\gamma S$.
\item
Let $\gamma$ have width $d-r+1$, and let $\tilde{\gamma}$ be the Young diagram obtained by deleting the first row of $\gamma$. Then there is an exact triangle
$$  \begin{tikzcd}\bbS^\gamma H \rar &  \pi_* i_P^* \bbS^\gamma S \rar & \bbS^{\tilde{\gamma}} H\otimes \det H^\vee [r-d] \rar[dashed] & {} \end{tikzcd}$$
in $D^b(Q)$.

\end{enumerate}
\end{prop}

\begin{rem}If $d>r$ then \textit{(iii)} is the statement that the non-zero higher push-down sheaves of $i_P^*\bbS^\gamma S$ are
$$\RDerived^k\pi_* i_P^*\bbS^\gamma S  =\left\{ \begin{array}{cl} \bbS^\gamma H \hspace{1cm} & k=0, \\
 \bbS^{\tilde{\gamma}} H\otimes \det H^\vee \hspace{1cm}&  k=d-r. \end{array}\right.$$
However if $d=r$ then we may get a non-split extension of bundles on $Q$.
\end{rem}

\begin{proof}
 On $\mfP$ we have a short exact sequence of bundles
$$ 0 \To H \To S \To L \To 0.$$
Thus $\bbS^\gamma S$ has a filtration whose associated graded pieces are
\begin{equation}\label{equation.schurfiltration}\bigoplus_{\alpha}\left( \bbS^\alpha H \otimes  L^{\otimes t} \right)^{\oplus c_{\alpha, \tau}^{\gamma}} \end{equation}
where $\tau$ is a partition of width $t$ and height 1, and $c_{\alpha, \tau}^{\gamma}$ are the Littlewood-Richardson coefficients \cite{Fulton:1996tk}. This means we can compute $\pi_*\bbS^\gamma S$ and $\pi_* i_P^* \bbS^\gamma S$ by spectral sequences that start with the push-downs of this graded bundle. Thus what we need to calculate is $\pi_* (L^{\otimes t})$ and $\pi_* i_P^*( L^{\otimes t}) $ for $t\geq0$.

Fix a point  $q\in\mfQ$, the fibre $\mfP_q$ over this point is the affine stack \eqref{equation.fibreofP}. The restriction of $L$ to $\mfP_q$ is the negative tautological line bundle, so any positive powers of it have no (derived) global sections. Since this is true at all fibres it implies that $\pi_* (L^{\otimes t}) = 0$ for $t>0$. Similarly $\pi_* (\cO) = \cO$. The fibres of $P$ on the other hand are projective spaces, so we still have  $\pi_* i_P^*\cO = \cO$ and $\pi_* i_P^*(L^{\otimes t}) = 0$ for $0<t<d-r+1$, but when $t\geq d-r+1$ we have a single higher push-down sheaf. In particular we have $$\pi_*i_P^*(L^{\otimes d-r+1}) = \det H^\vee [r-d].$$ 

\begin{enumerate}
\item Apply $\pi_*$ to \eqref{equation.schurfiltration}. Only the degree-zero piece $\bbS^\gamma H$ survives, and the spectral sequence collapses.

\item By the width restriction on $\gamma$, no powers of $L$ above $L^{\otimes d-r}$ occur in \eqref{equation.schurfiltration}, so when we apply $\pi_*i_P^*$
again only the piece $\bbS^\gamma H$ survives.

\item By the Littlewood-Richardson rule (or the simpler Pieri rule \cite{Fulton:1996tk}), the degree $d-r+1$ piece of \eqref{equation.schurfiltration} is $\bbS^{\tilde{\gamma}} H \otimes L^{\otimes d-r+1}$, and there are no pieces of higher degree. So when we apply $\pi_*i_P^*$ we get two surviving terms, the spectral sequence collapses, and $\pi_* i_P^* \bbS^\gamma S$ is an extension as claimed.
\end{enumerate}\end{proof}

We now say a little more about the second map that occurs in the exact triangle in (iii). This is used in the proof of Lemma \ref{lemma.cotwisttransferproperty3}. Let $\gamma$ have width $d-r+1$. We observed that the filtration \eqref{equation.schurfiltration} concludes with a natural map
 $$q: \bbS^\gamma S \To \bbS^{\tilde{\gamma}} H \otimes L^{\otimes d-r+1} $$
where as before $\tilde{\gamma}$ is $\gamma$ with its first row removed. The nature of this map is clearer if we switch to Schur powers of the dual bundles. Let \begin{eqnarray*} \delta &=& \comp{d-r+1}{r}{\gamma} \\ &=& \comp{d-r+1}{r-1}{\tilde{\gamma}}.\end{eqnarray*} Then 
$$\bbS^\gamma S  = \bbS^\delta S^\vee \otimes (\det S)^{r-d-1} $$
and 
\begin{eqnarray*} \bbS^{\tilde{\gamma}} H \otimes L^{\otimes d-r+1} &=& \bbS^\delta H^\vee\otimes (\det H)^{r-d-1} \otimes L^{\otimes d-r+1} \\ & =& \bbS^\delta H^\vee\otimes (\det S)^{r-d-1}.\end{eqnarray*}
The map $q$ is the tautological map from $\bbS^\delta S^\vee$ to $\bbS^\delta H^\vee$, twisted by a line bundle. We also have that
$$ \bbS^{\tilde{\gamma}} H \otimes L^{\otimes d-r+1} = \pi^!\left(   \bbS^{\tilde{\gamma}} H \otimes \det H^\vee [r-d]\right).$$
Now restrict to the space $P$, and use the $\pi_*$-$\pi^!$ adjunction. The map $q$ gets sent to the map
$$ \pi_* i_P^*\bbS^\gamma S \To \bbS^{\tilde{\gamma}} H\otimes \det H^\vee[r-d] $$
that occurs in the statement of (iii).

\subsection{The right-hand square}
\label{section.righthand_square}
 We consider the affine stack
$$\mfT = [\Hom(V, S) \; / \; \operatorname{GL}(S) ]. $$
We also consider a second stack
$$ \bar{\mfS} =[\Hom(V, H)\oplus \Hom(H,S) \; / \; \operatorname{GL}(H)\times \operatorname{GL}(S)],$$
and let $\mfS\subset \bar{\mfS}$ be the open substack where the map from $H$ to $S$ is an injection. We let $j$ be the map
$$j: \mfS \To \mfT $$
given by composing the two factors and forgetting the $\operatorname{GL}(H)$ action. 

As in the body of the paper, we write $\cO(1):=\det S^\vee$, and $\cO\br{1}:= \det H^\vee$. Then
\begin{equation}\label{equation.j^!2} j^!(-) = j^*(-)\otimes K_j[\dim j] = j^*(-)(r-d-1)\br{d-r}[r-d-1] \end{equation}
(recalling that $\det V$ is trivialized), which of course agrees with \eqref{equation.j^!}.

 The image of $j$ is the degenerate locus in $\mfT$ where the rank of the linear map has dropped. More specifically, if we fix a point $t\in \mfT$ then we have a vector space $C$ defined as the cokernel
$$V \To S \To C \To 0.$$
Generically this will be zero-dimensional, and it will jump up in dimension for non-generic $t$. The fibre of $\mfS$ over $t$ is the projective space $\mfS_t=\P^\vee C$ of hyperplanes in $C$.

\begin{lem}\label{lemma.nohigherj_*}
Let $\delta$ be any Young diagram. Then $j_*\bbS^\delta H^\vee$ is just a sheaf, i.e. there are no higher push-down sheaves. 
\end{lem}
\begin{proof}
Pick $t\in \Hom(V,S)$. The restriction of $H$ to the fibre $\mfS_t$ is isomorphic to
$$\tilde{H} \oplus \cO^{\oplus rk(t)} $$
where $\tilde{H}$ is the tautological subbundle on $\P^\vee C$. Thus the restriction of $\bbS^\delta H^\vee$ to $\mfS_t$ is a non-negative bundle, and has no higher cohomology. Since this is true at all fibres, the higher push-down sheaves vanish.
\end{proof}

We will construct a locally-free resolution of the torsion sheaf $j_*\bbS^\delta H^\vee$, for certain $\delta$. In order to describe this resolution, we first need to introduce some combinatorics with Young diagrams.

\begin{alg}\label{algorithm}
Let $\delta$ be a Young diagram of height $<r$. We define a sequence of Young diagrams $\delta_1, \delta_2,\ldots$ 
starting from $\delta_0:=\delta$, by the following procedure:
\begin{itemize}
\item $\delta_1$ is obtained from $\delta_0$ by adding boxes to the first column until it reaches height $r$.
\item $\delta_k$ is obtained from $\delta_{k-1}$ by adding boxes to the $k^{\text{th}}$ column, until its height is one more than the height of the $(k-1)^{\text{th}}$ column of $\delta_0$. 
\end{itemize}
We let
$s_k$ denote the total number of boxes added at stage $k$, i.e. $s_k$ is the difference in size between $\delta_k$ and $\delta_0$.
\end{alg}

\begin{rem}In Algorithm \ref{algorithm}, the last box added at stage $k$ is immediately to the right of the first box added at stage $k-1$.
\end{rem}

\begin{lem}\label{lemma.staircase_Young_diag_components}Writing $h_k$ for the height of the $k^{\text{th}}$ column of a Young diagram $\delta$ of height $<r$ given by $\delta = (\cpt{\delta}{1}, \ldots, \cpt{\delta}{r-1})$ we have \[\staircase{k} = (\cpt{\delta}{1}, \ldots, \cpt{\delta}{h_k}, k, \cpt{\delta}{h_k+1}+1, \ldots, \cpt{\delta}{r-1}+1).\]

\begin{proof} Induction. \end{proof}
\end{lem}

Now we can state the following theorem, whose proof is given in Section \ref{section.requiredresolutions}:
\begin{thm}\label{theorem.freeresolutions}
Let $\delta$ be a Young diagram of height $<r$ and width $\leq d-r +1$. We have an exact sequence of sheaves on $\mfT$
\begin{equation*} \begin{tikzcd}[column sep=14pt] 0 \rar &  \bbS^{\delta_{K}} S^\vee \otimes \Wedge^{s_{K}} V  \rar &\;   \ldots\;  \rar &\bbS^{\delta_1} S^\vee \otimes \Wedge^{s_1} V  \rar &   \bbS^{\delta_0} S^\vee \rar & j_*\bbS^\delta H^\vee \rar & 0 \end{tikzcd} \end{equation*} 
where $K=d-r+1$, and the $\delta_k$ and $s_k$ are defined in Algorithm \ref{algorithm}.
\end{thm}

We give some simple examples showing how Theorem \ref{theorem.freeresolutions} reproduces certain exact sequences used in Section \ref{section.intro_grassmannian}.

\begin{eg}\label{eg.eagon-northcott} Set $r=2$, $d=4$, and let $\delta$ be the empty partition. The partitions $\delta_k$ and associated resolution are shown in Figure \ref{figure.eagon-northcott}. Restricting the resolution to the full rank locus of $\mcT$, we obtain a long exact sequence which is the Eagon-Northcott complex used in Section \ref{section.intro_grassmannian}.
\end{eg}
\begin{figure}[H]
\begin{center}
\begin{tikzpicture}[auto,>=latex',scale=0.6]
\YoungDiag{en_delta_0}{0}{0}{3}{2}{{0}}{$\delta=\delta_0$}
\YoungDiag{en_delta_1}{5}{0}{3}{2}{{1,1}}{$\delta_1$}
\YoungDiag{en_delta_2}{10}{0}{3}{2}{{2,1}}{$\delta_2$}
\YoungDiag{en_delta_3}{15}{0}{3}{2}{{3,1}}{$\delta_3$}

\draw[dashed, line width=0.8pt] (0,2) -- (1,2) -- (1,1) -- (3,1) -- (3,0);
\end{tikzpicture}
\[\begin{tikzcd}[column sep=13pt] 0 \rar & \Sym^2 S^\vee(1) \otimes \Wedge^4 V \rar & S^\vee(1) \otimes \Wedge^3 V \rar & \cO(1) \otimes \Wedge^2 V \rar & \cO \rar & j_*\cO \rar & 0 \end{tikzcd}\]
\end{center}
\caption{}
\label{figure.eagon-northcott}
\end{figure}

\begin{eg}\label{eg.buchsbaum-rim} Set $r=2$, $d=4$ again, and let $\delta=(1,0)$. The partitions and resolution are shown in Figure \ref{figure.buchsbaum-rim}. Restricting  to the full rank locus of $\mcT$, we obtain the Buchsbaum-Rim complex used in Section~\ref{section.intro_grassmannian}.
\end{eg}
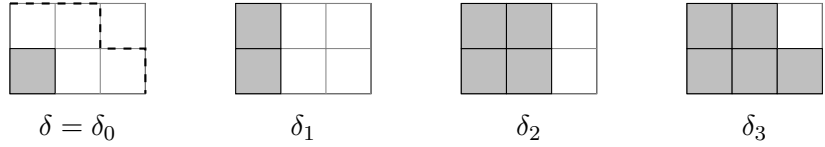
\begin{figure}[H]
\begin{center}
\begin{tikzpicture}[auto,>=latex',scale=0.6]
\YoungDiag{br_delta_0}{0}{0}{3}{2}{{1}}{$\delta=\delta_0$}
\YoungDiag{br_delta_1}{5}{0}{3}{2}{{1,1}}{$\delta_1$}
\YoungDiag{br_delta_2}{10}{0}{3}{2}{{2,2}}{$\delta_2$}
\YoungDiag{br_delta_3}{15}{0}{3}{2}{{3,2}}{$\delta_3$}

\draw[dashed, line width=0.8pt] (0,2) -- (2,2) -- (2,1) -- (3,1) -- (3,0);
\end{tikzpicture}
\[ \begin{tikzcd}[column sep=17pt] 0 \rar & S^\vee(2) \otimes \Wedge^4 V \rar & \cO(2) \otimes \Wedge^3 V \rar & \cO(1) \otimes V \rar & S^\vee \rar & j_* H^\vee \rar & 0 \end{tikzcd} \]
\end{center}
\caption{}
\label{figure.buchsbaum-rim}
\end{figure}

We make a few elementary observations on the combinatorics: these all follow from the size restrictions on $\delta$.

\begin{rem}\phantomsection\label{remark.combinatorics}

\begin{enumerate}
\item We have $s_{K} \leq d$ but $s_{K+1} > d$, which explains (at least formally) why the resolution terminates at $K$ terms. 
\item The height of $\delta_k$ is $r$ for $k>0$, and $<r$ for $k=0$. The width of $\delta_k$ is $\leq d-r+1$  for all $k\leq K$. 
\end{enumerate}
\smallskip \noindent Additionally, if the width of $\delta$ is $<d-r+1$ then:
\begin{enumerate}\addtocounter{enumi}{2}
\item $s_K=d$.
\item The width of $\delta_k$ is $<d-r+1$ for $k<K$, and the width of $\delta_K$ is $d-r+1$.
\item If we delete the first row and the first column from $\delta_K$ then we recover the diagram $\delta$.
\end{enumerate}
\smallskip \noindent On the other hand if the width of $\delta$ is equal to $d-r+1$ then:
\begin{enumerate}\addtocounter{enumi}{5}
\item The width of $\delta_k$ is $d-r+1$ for all $k\leq K$.
\end{enumerate}
\end{rem}

Remark (ii) implies that we can define
\begin{equation}\label{equation.epsilon_defn}\epsilon_k = \comp{d-r+1}{r}{\delta_k}\end{equation}
for $0\leq k \leq K$. Then the following corollary is immediate using \eqref{equation.j^!2}.
\begin{cor}\label{corollary.j^!j_*}
For $\delta$ as before in Theorem \ref{theorem.freeresolutions}, $j^!j_* \bbS^\delta H^\vee$ is quasi-isomorphic to the complex
\begin{equation*} \complex{ \begin{tikzcd}[column sep=14pt] \underline{\bbS^{\epsilon_{K}} S \br{d-r} \otimes \Wedge^{s_{K}} V} \rar \nextcell\;   \ldots\;  \rar \nextcell  \bbS^{\epsilon_1} S \br{d-r} \otimes \Wedge^{s_1} V \rar \nextcell   \bbS^{\epsilon_0} S\br{d-r} \end{tikzcd} }. \end{equation*}
\end{cor}

Taking complements in Remark \ref{remark.combinatorics} tells us the following:

\begin{rem}\phantomsection\label{remark.combinatoricsdual}

\begin{list}{(\roman{enumi}*)}{\usecounter{enumi}}
\addtocounter{enumi}{1}
\item The width of $\epsilon_k$ is $d-r+1$ for $k=0$, and $<d-r+1$ for $k>0$. The height of $\epsilon_k$ is $\leq r$ for all $k$.
\end{list}
\smallskip \noindent Also if the width of $\delta$ is $<d-r+1$ then:
\begin{list}{(\roman{enumi}*)}{\usecounter{enumi}}
\addtocounter{enumi}{2}
\item $s_K=d$.
\item  The height of $\epsilon_k$ is $r$ for $k<K$, and the height of $\epsilon_K$ is $<r$.
\item $\epsilon_K = \comp{d-r}{r-1}{\delta}$.
\end{list}
\smallskip \noindent If the width of $\delta$ is equal to $d-r+1$ then:
\begin{list}{(\roman{enumi}*)}{\usecounter{enumi}}
\addtocounter{enumi}{5}
\item  The height of $\epsilon_k$ is $\leq r-1$ for all $k$. 
\end{list}
\end{rem}

We end this section with some observations on the spaces of maps between various bundles on $\mfS$ and $\mfT$.

\begin{lem}\label{lemma.1dmappingspaces}
For any partition $\delta$, we have
\begin{enumerate}
\item
$$\Ext^0_{\mfS}(\bbS^\delta S, \bbS^\delta S) = \C $$
i.e. any map from this bundle to itself is a scalar multiple of the identity.
\item
$$\Ext^0_{\mfS}(\bbS^\delta H, \bbS^\delta S) = \C $$
i.e. any map between these two bundles is a scalar multiple of the tautological map.
\end{enumerate}
\end{lem}
\begin{proof}
We can work on $\bar{\mfS}$ instead, since the complement of $\mfS$ has codimension $\geq 2$. Then for (i) we just have to compute the $\operatorname{GL}(H)\times \operatorname{GL}(S)$ invariants in 
$$\bbS^\delta S^\vee\otimes \bbS^\delta S \otimes \Sym\left( V\otimes H^\vee \oplus H\otimes S^\vee \right).$$
This is an easy calculation using the Littlewood-Richardson rule and the facts that for any vector spaces $A$ and $B$ we have
$$\Sym(A\otimes B) = \bigoplus_{\lambda, \mu}  \bbS^\lambda A \otimes  \bbS^\mu B $$
\cite[Theorem 2.3.2]{Weyman:2007wu} and
$$( \bbS^\lambda A^\vee \otimes  \bbS^\mu A)^{\operatorname{GL}(A)} = \left\{ \begin{array}{cc}
\C\hspace{1cm} & \lambda = \mu, \\  0\hspace{1cm} & \lambda\neq \mu. \end{array} \right. $$
Part (ii) is identical.
\end{proof}

\begin{lem}\label{lemma.1dmappingspaces2}
Let $\delta_0,\delta_1,\ldots$ be a sequence of Young diagrams constructed by Algorithm \ref{algorithm} above. Then for any $k$,
$$\Hom_{\mfT}\left( \bbS^{\delta_{k+1}}S^\vee \otimes \Wedge^{s_{k+1}} V,\; \bbS^{\delta_{k}}S^\vee \otimes \Wedge^{s_{k}} V\right)^{\operatorname{SL}(V)} = \C, $$
i.e. the maps in the sequence in Theorem \ref{theorem.freeresolutions} are determined (up to scalar multiples) by the requirement of $\operatorname{SL}(V)$-equivariance.
\end{lem}
\begin{proof}
The calculation is very similar to the ones in Lemma \ref{lemma.1dmappingspaces}. Note that the Littlewood-Richardson coefficient $c_{\lambda, \delta_k}^{\delta_{k+1}}$ is zero unless $\lambda$ is the \quotes{column} $(1,1,\ldots,1)$ of height $s_{k+1}-s_k$, in which case it equals 1.
\end{proof}
We do not actually use Lemma \ref{lemma.1dmappingspaces2}, but it is interesting to note. It can also be considered a warm-up for the next lemma, which is more technical and is used in the proof of Lemma \ref{lemma.cotwisttransferproperty1}.
\begin{lem}\label{lemma.1dmappingspaces3}
Let $\delta$ have height $<r$ and width equal to $d-r+1$, and let $\delta_0,\delta_1,\ldots$ be the corresponding sequence of Young diagrams. As above, let $K=d-r+1$, and $\epsilon_K = \comp{d-r+1}{r}{\delta_K}$. Then 
$$\Ext^0_{\mfS}\left( \bbS^{\delta}H^\vee ,\; \bbS^{\epsilon_{K}}S\br{d-r}\otimes \Wedge^{s_{K}} V \right)^{\operatorname{SL}(V)} = \C.$$
\end{lem}
\begin{proof}
As in Lemma \ref{lemma.1dmappingspaces} we can work on $\bar{\mfS}$, and this is a computation of invariants. After taking $\operatorname{GL}(S)$-invariants, we are left with
\begin{equation}\label{eqn.afterGLSinvariants}\bigoplus_{\lambda} \bbS^{\delta} H\otimes \bbS^{\epsilon_K} H\otimes \bbS^{\lambda} H^\vee\br{d-r}\otimes \bbS^{\lambda}V\otimes \Wedge^{s_K} V.\end{equation}
Now consider the expression
\begin{eqnarray*}\bbS^{\delta} H\otimes \bbS^{\epsilon_K} H\otimes\bbS^\lambda H^\vee\br{d-r} &=& \bbS^{\delta} H\br{d-r+1}\otimes \bbS^{\epsilon_K} H\br{-1}\otimes\bbS^\lambda H^\vee \\ &=& \bbS^{\tilde{\epsilon}_0} H^\vee\otimes \bbS^{\tilde{\epsilon}_K} H \otimes\bbS^\lambda H^\vee
\end{eqnarray*}
where  $\tilde{\epsilon}_0= \comp{d-r+1}{r-1}{\delta}$, and $\tilde{\epsilon}_K$ is the diagram obtained from $\epsilon_K$ by adding on an extra column of height $r-1$, which is well-defined because the height of $\epsilon_K$ is $\leq r-1$ (Remark \ref{remark.combinatoricsdual}(vi*)). Then the $\operatorname{GL}(H)$-invariants in \eqref{eqn.afterGLSinvariants} are
\begin{equation}\label{eqn.afterGLHinvariants}\bigoplus_{\lambda} \bbS^\lambda V\otimes \Wedge^{s_K} V^{\oplus c_{\tilde{\epsilon}_0, \lambda}^{\tilde{\epsilon}_K} }.\end{equation}
Now let $h_K$ be the height of the $K^{\text{th}}$ column of $\delta$, and use Lemma \ref{lemma.staircase_Young_diag_components} to deduce that
$$\tilde{\epsilon}_K = (K-\delta^{r-1}, \ldots, K-\delta^{h_K+1},1,\ldots,1)$$
where the number of rows is $r-1$. But by definition
$$\tilde{\epsilon}_0 = (K-\delta^{r-1},\ldots, K-\delta^{h_K+1})$$
and so the Littlewood-Richardson coefficient $c_{\tilde{\epsilon}_0, \lambda}^{\tilde{\epsilon}_K}$ is equal to 1 if $\lambda$ is the column $(1,\ldots,1)$ of height $h_k$, and equal to 0 otherwise. Hence \eqref{eqn.afterGLHinvariants} is just
$$\Wedge^{h_K}V\otimes \Wedge^{s_K} V $$
and this contains a 1-dimensional space of $\operatorname{SL}(V)$-invariants, since $h_K+s_K=d$.
\end{proof}

\subsection{Locally-free resolutions}\label{section.requiredresolutions}
We now prove Theorem \ref{theorem.freeresolutions}, which turns out to be an extreme case of the twisted Lascoux resolution \cite[Section 6.1]{Weyman:2007wu}. Weyman gives this resolution implicitly: we present a Borel--Weil--Bott calculation that makes it explicit, as required for our purposes. We could not find this given in the literature, but note that \cite{Fonarev:2011tq} uses the same combinatorics to produce exact sequences on Grassmannians. We briefly review Weyman's construction, slightly modifying his language with the aim of providing a bridge between his account and our application.

Let $G$ be a linearly reductive group, $P$ a parabolic subgroup of $G$, and $T$ a vector space with a $G$-action. We also choose a subspace $U$ of $T$ with a compatible $P$-action. Consider then a diagram
\[\begin{tikzcd}[row sep=25pt] U\times_P G \drar{j} \dar &
\\
\Im j \rar[hook] & T \end{tikzcd} \]
where $j$ takes $(u,g) \longmapsto (gu)$. Say we are interested in obtaining resolutions of torsion sheaves on $T$, supported on $\Im j$, which are obtained as direct images under the map $j$: \cite{Weyman:2007wu} uses this setup to calculate syzygies on determinantal varieties, and it turns out to be what we require also. Following {\it loc. cit.}, Sections 5.1 and 5.4, we form a diagram as follows, factoring $j$ into an embedding $i$ and a flat projection $q$:
\begin{equation}\label{equation.Weyman_diagram}\begin{tikzcd}[row sep=25pt] U\times_P G \rar[hook]{i} \drar[swap]{j} & T\times G/P \dar{q} \rar{p} & G/P
\\
 & T \end{tikzcd} \end{equation}
Here $i$ takes $(u,g) \longmapsto (gu,gP)$.

The relevant result from Weyman is then:
\begin{thm}\emph{\cite[Theorem 5.4.1]{Weyman:2007wu}}\label{thm.Weyman} Take a vector bundle $\cE$ on $G/P$ induced from a representation of $P$, and assume that $j_* i^* p^* \cE$ is a sheaf on $T$ (i.e. it has no higher push-downs). Then this sheaf has a $G$-equivariant resolution given by $\mcF_\bullet$, where \[\mcF_k := \bigoplus_{j \geq 0} \RDerived^j q_* \left(\Wedge^{k+j} ((T/U)^\vee) \otimes p^* \cE \right).\] 
\end{thm}

\begin{rem} The proof uses the commutativity of \eqref{equation.Weyman_diagram} and the projection formula to rewrite the sheaf in question as \[ j_* i^* p^* \cE = q_* i_* i^* p^* \cE = q_* ( \cO_{\Im i} \otimes p^* \cE) \]
and then evaluates this using the Koszul resolution of $\cO_{\Im i}$. The relevant spectral sequence simplifies because of equivariance considerations \cite[Section 5.2]{Weyman:2007wu}.
\end{rem}

\begin{rem} Note that the bundle $\cE$ on $G/P$ is the `twist' in the twisted Lascoux resolution noted above.\end{rem}

Now we are ready for:

\begin{proof}[Proof of Theorem \ref{theorem.freeresolutions}]
We take $G=\operatorname{GL}(S)$, choose an inclusion $H \into S$, and define $P$ as the parabolic preserving $H$. Then we take $T$ to be underlying vector space of our stack $\mfT$, i.e. $T:=\Hom(V,S)$ with its $G$-action. $T$ acquires a $P$-action, compatible with the natural $P$-action on $U:=\Hom(V,H)$.

We apply the construction explained above to obtain a diagram
\[\begin{tikzcd}[row sep=25pt]Z \rar[hook]{i} \drar[swap]{j} & \Hom(V,S)\times G/P \dar{q} \rar{p} & G/P
\\
 & \Hom(V,S) \end{tikzcd}\]
where $Z$ is $\Hom(V,H) \times_P G$.

We now just need to reinterpret this in our stacky language: in particular we relate $Z$ and our stack $\cS$ from Section \ref{section.righthand_square}. Firstly we have an equivalence of categories of sheaves on the variety $G/P$ and on the open substack of \[[\Hom(H,S) \; / \; GL(H)]\] where the map from $H$ to $S$ is an injection (these are two alternative descriptions of the projective space $\P^\vee S$). It follows that we have a similar equivalence relating $Z$ and the open substack of \[[\Hom(V,H)\oplus\Hom(H,S)\; / \;\operatorname{GL}(H)]\] where the map from $H$ to $S$ is an injection. Hence we see that working $G$-equivariantly with the morphism $j$ in \eqref{equation.Weyman_diagram} is just the same as working with the morphism $j : \cS \to \mfT$ as defined in Section \ref{section.righthand_square}.

Now we consider the representation $\bbS^\delta H^\vee$ of $P$. This induces a sheaf on $G/P$ and also on $\mfS$, as appears in the statement of Theorem \ref{theorem.freeresolutions}. We want then to resolve $j_* i^* p^* \bbS^\delta H^\vee$. Lemma \ref{lemma.nohigherj_*} gives that there are no higher push-down sheaves, so Theorem \ref{thm.Weyman} immediately gives us a $G$-equivariant resolution $\mcF_\bullet$ on $T=\Hom(V,S)$ where \[\mcF_k := \bigoplus_{j \geq 0} \RDerived^j q_* \left(\Wedge^{k+j} (V \otimes (S/H)^\vee) \otimes \bbS^\delta H^\vee\right).\]
This yields the required resolution on $\mfT$. Lemma \ref{lem.staircase} below shows that these pushforwards evaluate to the terms given in the statement of Theorem \ref{theorem.freeresolutions} in the Section \ref{section.righthand_square}.
\end{proof}

\begin{lem}\label{lem.staircase}For $0 \leq k \leq K:=d-r+1$ we have \[\mcF_k = \bbS^{\staircase{k}} S^\vee \otimes \Wedge^{s_{k}} V \] where $\staircase{k}$ and $s_k$ are defined in Algorithm \ref{algorithm}, and $\mcF_k = 0$ otherwise.

\begin{proof} Rearranging to give \[\mcF_k := \bigoplus_{i \geq k} \RDerived^{i-k}q_* \left((S/H)^{\vee i} \otimes \bbS^\delta H^\vee\right)\otimes \wedge^i V\] it suffices to work fibrewise and evaluate \[H^\bullet \left(\operatorname{GL}(S)/P, \:(S/H)^{\vee i} \otimes \bbS^\delta H^\vee\right).\] We explain how to calculate this cohomology group explicitly using the Borel--Weil--Bott theorem \cite[Section 4.1]{Weyman:2007wu}. According to the standard prescription (see for example {\it loc. cit.}, Corollary 4.1.9) the bundle in question corresponds to a $\operatorname{GL}(S)$-weight \[\alpha(i) := (\cpt{\delta}{1}, \ldots, \cpt{\delta}{r-1}, i).\] First we note that if $\alpha=\alpha(i)$ is \emph{dominant} (i.e. given by a sequence of non-increasing integers) then we have $H^0 = \bbS^\alpha S^\vee$ and $H^{>0} = 0$. More generally, the theorem determines the cohomology, which occurs in at most one degree, according to the behaviour of the weight $\alpha$ under the twisted action of the Weyl group $W=S_r$. For $w \in S_r$ this action is given by \[w \bullet \alpha := w(\alpha + \rho)-\rho\] where $\rho:=(r,\ldots,2,1)$. We say that $\alpha$ is \emphasis{regular} if there exists a unique $w \in S_r$ such that $w \bullet \alpha$ is dominant. We then have three mutually exclusive cases, with the theorem giving the cohomology in each:

\begin{tabular}{llll}
(1) & $\alpha$ dominant & $\so$ & $H^0 = \bbS^\alpha S^\vee$ \\
(2) & $\alpha$ regular, non-dominant & $\so$ & $H^l = \bbS^{w \bullet \alpha} S^\vee$, \:$l = \operatorname{length}(w)$ \\
(3) & $\alpha$ non-regular & $\so$ & $H^\bullet = 0$
\end{tabular}

As $i$ varies we classify the weight $\alpha(i)$ as follows:

\begin{itemize}
\item \case{1} $\alpha(i)$ \emphasis{dominant} if $0 \leq i \leq \cpt{\delta}{r-1}$.

This is immediate: dominant $\operatorname{GL}(S)$-weights correspond precisely to non-increasing integer sequences.

\item \case{2} $\alpha(i)$ \emphasis{regular, non-dominant} if there exists a natural number $l \leq r-1$ such that \begin{equation}\label{eqn.case_2_condition}\cpt{\delta}{r-l} < i - l \leq \cpt{\delta}{r-l-1}.\end{equation} (Here for convenience we set $\cpt{\delta}{0}=\infty$ so that when $l=r-1$ the second inequality is redundant.)

In this case the cycle $w=(r-l \: \: \ldots \: \: r)$ gives \[w \bullet \alpha(i) = (\cpt{\delta}1, \ldots, \cpt{\delta}{r-l-1}, i-l, \cpt{\delta}{r-l}+1, \ldots, \cpt{\delta}{r-1}+1)\] which is dominant by the condition \eqref{eqn.case_2_condition}.

Now the crucial point is to observe that in fact $w \bullet \alpha(i) = \staircase{i-l}$, one of the Young diagrams obtained by applying Algorithm \ref{algorithm}. This follows from the description of the $\staircase{k}$ in Lemma \ref{lemma.staircase_Young_diag_components} after noting that the height of the $(i-l)^{\text{th}}$ column of $w \bullet \alpha(i)$ is given by $r-l-1$  by the condition \eqref{eqn.case_2_condition}.

We also observe that $s_{i-l}=i$, because $w \bullet \alpha(i) = \staircase{i-l}$ has the same number of boxes as $\alpha(i)$, which is $i$ more than the number of boxes in $\delta$. 

\item \case{3} $\alpha(i)$ \emphasis{non-regular} if there exists a natural number $l < r-1$ such that \begin{equation}\label{eqn.case_3_condition} i-l = \cpt{\delta}{r-l}.\end{equation}

In this case the transposition exchanging $r-l$ and $r$ stabilises $\alpha(i)$ under the twisted Weyl group action, because \[(\alpha(i)+\rho)^{r-l} := \cpt{\delta}{r-l} + l + 1 = i + 1 =: (\alpha(i)+\rho)^r \] by the condition \eqref{eqn.case_3_condition}.
\end{itemize}

In summary we have that if $\alpha(i)$ is regular (the first two cases), then there exists a (possibly trivial) Weyl group element $w(i)$ with length $l(i)$ such that

\begin{enumerate}
\item $w(i)\bullet \alpha(i) = \staircase{i-l(i)}$,
\item$s_{i-l(i)} = i$.
\end{enumerate}

We see then that $\alpha(i)$ contributes to $\mcF_\bullet$ via homology in degree $l(i)$ and thence to $\mcF_k$ when $i-k=l(i)$. This occurs precisely when $k=i-l(i)$, and for each $k$ this equation has a unique solution for $i$, as non-uniqueness would contradict $s_{i-l(i)} = i$. Hence we deduce that
\[\mcF_k = \RDerived^{l(i)}q_* \left((S/H)^{\vee i} \otimes \bbS^\delta H^\vee\right) \otimes \wedge^i V.\]

The required push-down then comes from
\begin{align*}H^{l(i)} \left(\operatorname{GL}(S)/P, \:(S/H)^{\vee i} \otimes \bbS^\delta H^\vee\right) & = \bbS^{w(i) \bullet \alpha(i)} S^\vee \\
& = \bbS^{\staircase{i-l(i)}}  S^\vee \\
& = \bbS^{\staircase{k}}  S^\vee \end{align*} and noting that $i = s_{i-l(i)} = s_k$ gives the result.\end{proof}
\end{lem}

\begin{eg}We illustrate in Figure \ref{fig.staircase-proof} how the 3 cases in the proof of Lemma \ref{lem.staircase} occur in the example $\delta=(3,1)$ with $r=3$. We give diagrams corresponding to the $\operatorname{GL}(S)$-weights $\alpha(0), \ldots, \alpha(5)$, with the row lengths of the diagrams corresponding to components of the respective weights. Note that in this example $\alpha(6)$, $\alpha(7)$, $\ldots$ are regular with $l=2$.\end{eg}

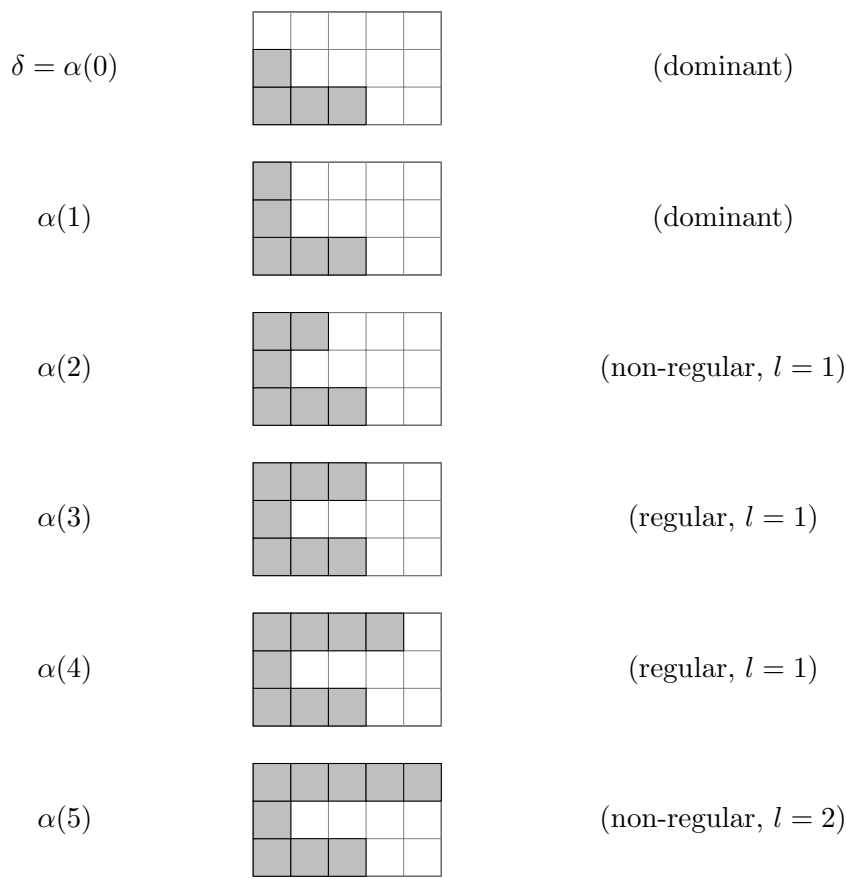
\begin{figure}[h]
\begin{center}
\begin{tikzpicture}[auto,>=latex',scale=0.5]
\AnnotatedYoungDiag{alpha_0}{0}{0}{5}{3}{{3,1}}{$\delta=\alpha(0)$}{(dominant)}{2}
\AnnotatedYoungDiag{alpha_1}{0}{-4}{5}{3}{{3,1,1}}{$\alpha(1)$}{(dominant)}{2}
\AnnotatedYoungDiag{alpha_2}{0}{-8}{5}{3}{{3,1,2}}{$\alpha(2)$}{(non-regular, $l=1$)}{2}
\AnnotatedYoungDiag{alpha_3}{0}{-12}{5}{3}{{3,1,3}}{$\alpha(3)$}{(regular, $l=1$)}{2}
\AnnotatedYoungDiag{alpha_4}{0}{-16}{5}{3}{{3,1,4}}{$\alpha(4)$}{(regular, $l=1$)}{2}
\AnnotatedYoungDiag{alpha_5}{0}{-20}{5}{3}{{3,1,5}}{$\alpha(5)$}{(non-regular, $l=2$)}{2}
\end{tikzpicture}
\end{center}
\caption{Cases arising for $\alpha(i)$ in proof of Lemma \ref{lem.staircase}.}
\label{fig.staircase-proof}
\end{figure}




\end{document}